\documentclass[11pt]{amsart}
\usepackage{graphicx}
\usepackage{graphics}
\usepackage{amsmath}
\usepackage{amssymb}
\usepackage{amscd}
\usepackage{latexsym}
\usepackage[all]{xy}
\begin{document}
\textwidth 5.5in
\textheight 8.3in
\evensidemargin .75in
\oddsidemargin.75in

\newtheorem{lem}{Lemma}
\newtheorem{conj}{Conjecture}
\newtheorem{defn}{Definition}
\newtheorem{thm}{Theorem}
\newtheorem{cor}{Corollary}
\newtheorem{lis}{List}
\newtheorem{exm}{Example}
\newtheorem{prob}{Problem}
\newtheorem{rmk}{Remark}
\newtheorem{que}{Question}
\newtheorem{prop}{Proposition}
\newtheorem{clm}{Claim}
\newcommand{\p}[3]{\Phi_{p,#1}^{#2}(#3)}
\def\Z{\mathbb Z}
\def\R{\mathbb R}
\def\g{\overline{g}}
\def\odots{\reflectbox{\text{$\ddots$}}}
\newcommand{\tg}{\overline{g}}
\def\proof{{\bf Proof. }}
\def\ee{\epsilon_1'}
\def\ef{\epsilon_2'}
\title{Variations of 4-dimensional twists obtained by an infinite order plug}
\author{Motoo Tange}
\thanks{This work was supported by JSPS KAKENHI Grant Number 24-840006.}
\address{Institute of Mathematics, \\
University of Tsukuba, Ibaraki 305-8571 JAPAN }
\email{tange@math.tsukuba.ac.jp}
\date{\today}
\maketitle
\begin{abstract}
In the previous paper the author defined an infinite order plug $(P,\varphi)$ which gives rise to infinite
Fintushel-Stern's knot-surgeries.
Here, we give two 4-dimensional infinitely many exotic families $Y_n$, $Z_n$ of exotic enlargements of the plug.
The families $Y_n$, $Z_n$ have $b_2=3$, $4$ and the boundaries are 3-manifolds with $b_1=1$, $0$ respectively.
We give a plug (or g-cork) twist $(P,\varphi_{p,q})$ producing the 2-bridge knot or link surgery by combining the plug $(P,\varphi)$.
As a further example, we describe a 4-dimensional twist $(M,\mu)$ between knot-surgeries for two mutant knots.
The twisted double concerning $(M,\mu)$ gives a candidate of exotic $\#^2S^2\times S^2$.
\end{abstract}

\section{Introduction}
\label{intro}
\subsection{Corks and plugs}
\label{smoothstr}
If two smooth manifolds $X,X'$ are homeomorphic but non-diffeomorphic,
then we say that $X$ and $X'$ are {\it exotic (or exotic pair)}.

A cut-and-paste is a performance removing a submanifold $Z$ from $X$ and regluing $Y$ via $\phi:\partial Y\to \partial Z$.
We use the notation $(X-Z)\cup_\phi Y$ for the cut-and-paste.
We call a cut-and-paste {\it a local move} in this paper.
Let $Y$ be a (codimension 0) submanifold of a 4-manifold $X$.
Let $\phi$ be a diffeomorphism $\partial Y\to \partial Y$.
We denote the local move with respect to $(Y,\phi)$ by
$$X(Y,\phi):=[X-Y]\cup_{\phi}Y,$$
and call such a local move {\it a twist} $(Y,\phi)$.

For a pair of exotic 4-manifolds $X,X'$, we call a compact contractible Stein manifold $Cr$ {\it a cork}, if $Cr$ is smoothly embedded in $X$, and $X'$ is obtained by a cut-and-paste of $Cr\subset X$ according to a diffeomorphism $\tau:\partial Cr\to \partial Cr$.
Hence, the boundary diffeomorphism $\tau$ cannot extend to inside $Cr$ as a diffeomorphism.
We also call the deformation a {\it cork twist} $(Cr,\tau)$.
Suppose that $X$ and $X'$ are two exotic simply-connected closed oriented 4-manifolds.
Then they are changed to each other by a cork twist $(Cr,\tau)$ with an order 2 boundary diffeomorphism $\tau$ (\cite{[AM]}, \cite{[CFHS]}, \cite{[M]}).
Namely, this means
$$X'=X(Cr,\tau).$$
Hence, in some sense, the existence of such a cork $(Cr,\tau)$ causes 4-dimensional differential structures.
Akbulut and Yasui in \cite{[AY1]} defined another kind of twists, which are called {\it plug twists}, and which change smooth structures.
The definition is given in the later section.
A study of cork and plug should play a key role in understanding differential structures of 4-manifolds.

In this paper, we produce two types of infinitely many exotic enlargements of $P$.
The meaning of studying enlargements is to investigate to what extent the `exotic producer' like cork or plug can extend to 
a larger 4-manifold.
Putting a plug $(P,\varphi)$ defined in \cite{Tan1} and other deformations together, we give infinite variations of 4-dimensional (plug or g-cork) twist for rational tangle replacement.
In terms of local move of knot we give a 4-dimensional twist $(M,\mu)$ with respect to the knot mutation,
however since $M$ is not a Stein manifold, the twist is neither plug nor g-cork.
\section{The definitions, results, and brief proofs.}
In this section we give definitions appeared, and a sequence of results obtained in this paper.
We also give several proofs proven immediately.
\subsection{Infinite order cork and plug.}
Let $P$ denote a 4-manifold described in {\sc Figure}~\ref{P}.
A diffeomorphism $\varphi:\partial P\to \partial P$ is defined to be {\sc Figure}~\ref{varphi}.
\begin{rmk}
Throughout this paper, any unlabeled component in any diagrams of 4-manifolds or of 3-manifolds stands for
a $0$-framed 2-handle or a $0$-surgery respectively.
\end{rmk}
\begin{figure}[htbp]
\begin{center}
\includegraphics{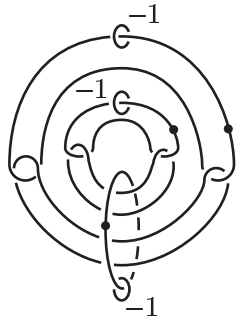}
\caption{A handle decomposition of $P$.}
\label{P}
\end{center}
\end{figure}
\begin{figure}[htbp]
\begin{center}
\includegraphics{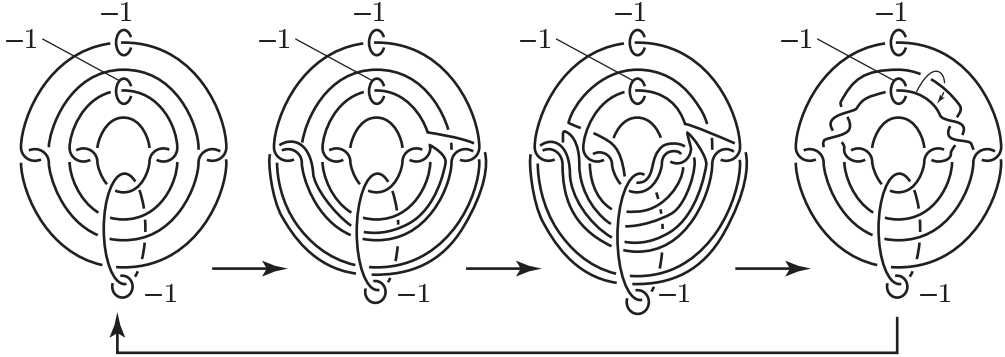}
\caption{A diffeomorphism $\varphi:\partial P\to \partial P$.}
\label{varphi}
\end{center}
\end{figure}
The paper \cite{Tan1} shows that the twist $(P,\varphi)$ is an infinite order plug,
and the square twist $(P,\varphi^2)$ is a generalized cork (a g-cork), as defined later.
Namely, $(P,\varphi)$ and $(P,\varphi^2)$ satisfy the following:
\begin{thm}[\cite{Tan1}]
\label{tangeprev}
$P$ is a Stein 4-manifold.
The map $\varphi:\partial P\to \partial P$ has infinite order and $\varphi$ cannot extend to a self-homeomorphism on inside $P$.
There exists a 4-manifold $X$ such that $\{X(P,\varphi^k)\}$ is a family of mutually exotic 4-manifolds.

The map $\varphi^2$ can extend to a self-homeomorphism,
but cannot extend to any self-diffeomorphism on $P$.
\end{thm}
In general, we define an infinite order plug, cork and g-cork.
\begin{defn}[Infinite order plug]
\label{pluginfinite}
$({\mathcal P},\phi)$ is an infinite order plug if it satisfies the following conditions:
\begin{enumerate}
\item ${\mathcal P}$ is a compact Stein 4-manifold.
\item $\phi$ cannot extend to a self-homeomorphism on ${\mathcal P}$.
\item There exists a 4-manifold $X$ and embedding ${\mathcal P}\subset X$ such that $\{X({\mathcal P},\varphi^k)\}$ is a family of mutually exotic 4-manifolds.
\end{enumerate}
\end{defn}
\begin{defn}[Infinite order cork]
\label{corkinfinite}
$({\mathcal C},\phi)$ is an infinite order cork if it satisfies the following conditions:
\begin{enumerate}
\item ${\mathcal C}$ is a compact contractible Stein 4-manifold.
\item $\phi^k$ cannot extend to any self-diffeomorphism on ${\mathcal C}$ for any positive integer $k$.
\end{enumerate}
If $X({\mathcal C},\phi^k)$ are mutually exotic 4-manifolds, then $({\mathcal C},\phi)$ is an infinite order cork for $\{X({\mathcal C},\phi^k)\}$.
In the case where ${\mathcal C}$ is not contractible, in place of being contractible in the (1) condition, we call $({\mathcal C},\phi)$ a generalized cork (or g-cork).
\end{defn}
The order of each $\phi$ in Definition~\ref{pluginfinite} and \ref{corkinfinite} as a mapping class on the boundary 3-manifold is infinite.
The twist $(P,\varphi)$ defined above is an infinite order plug and $(P,\varphi^2)$ is an infinite order g-cork.
\begin{figure}[htbp]
\begin{center}
\includegraphics{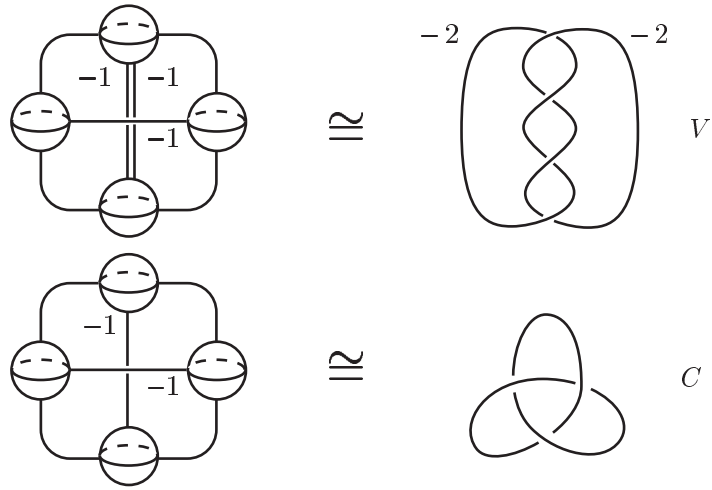}
\caption{The neighborhoods of Kodaira's singularity III and II (cusp).}
\label{KodIII}
\end{center}
\end{figure}
Furthermore, the plug twist $(P,\varphi)$ can make a Fintushel-Stern's knot-surgery.
Let $V$ and $C$ denote the neighborhoods of Kodaira's singularity III and II.
See {\sc Figure}~\ref{KodIII} for the diagrams.
$C$ is called a {\it cusp neighborhood}.
These diagrams can be also seen in \cite{KK}.
\begin{thm}[\cite{Tan1}]
\label{crossingchange}
Let $X$ be a 4-manifold containing $V$ and let $K$ be a knot.
Let $X_K$ be a knot-surgery of $X$ along the general fiber of $V$.
For a knot $K'$ obtained by changing a crossing of a diagram of $K$,
there exists an embedding $i:P\hookrightarrow X_K$ such that for the embedding $i$ we have
$$X_{K'}=X_K(P,\varphi).$$

The $n$-th power $(P,\varphi^n)$ makes an $n$ times full-twist 
$$X_{K_n}=X_{K}(P,\varphi^n),$$
where $K$ and $K_n$ are the two knots whose local diagrams are {\sc Figure}~\ref{tw} and whose remaining diagrams are the same thing.
\begin{figure}[htbp]
\begin{center}
\includegraphics{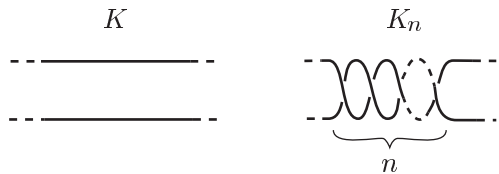}
\caption{$K_n$ is the $n$-full twist of $K$.}
\label{tw}
\end{center}
\end{figure}
\end{thm}
\begin{rmk}
The crossing change is a local move for knots or links.
Theorem~\ref{crossingchange} means that the plug twist $(P,\varphi)$ plays a role in `the crossing change of 4-manifolds' obtained by knot-surgery in some sense.
Similarly, for many other local moves of knots or links, one can construct a local move over 4-manifolds.
We will give an example of a 4-dimensional local move (twist) coming from a local move (knot mutation) of knots and links at the later section.
\end{rmk}

Here we define the knot-surgery and (2-component) link-surgery according to \cite{[FS]}.
Let $T\subset X$ be an embedded torus with trivial neighborhood and $K$ a knot in $S^3$.
The $0$-surgery $M_K$ cross $S^1$ naturally contains an embedded torus $T_m=\{\text{meridian}\}\times S^1$ with the trivial neighborhood.
Then, {\it (Fintushel-Stern's) knot-surgery} $X_K$ is defined to be the fiber-sum
$$X_K=(M_K\times S^1)\#_{T_m=T}X.$$

Let $U_1,U_2$ be two 4-manifolds containing an embedded tori $T_i\subset U_i$ with the trivial neighborhoods.
Let $L=K_1\cup K_2$ be a $2$-component link.
Let
$$\alpha_L:\pi_1(S^3-L)\to {\Bbb Z}$$
be a homomorphism satisfying $\alpha_L(m_i)=1$, where $m_i$ is the meridian curve of $K_i$.
Let $M_L$ be the $\alpha(\ell_i)$-surgery of $L$, where $\ell_i$ is the longitude of $K_i$.
Let $T_{m_i}$ be a torus $m_i\times S^1\subset M_L\times S^1$.
Then, we denote by $(U_1,U_2)_L$ the following double fiber-sum operation:
$$(U_1,U_2)_L=U_1\#_{T_1=T_{m_1}}(M_L\times S^1)\#_{T_{m_2}=T_2}U_2.$$
In the case of $U=U_1=U_2$, we write as $(U,U)_L=U_L$.
We call $(U_1,U_2)_L$ {\it the link-surgery by the link $L$}.
\subsection{Two kinds of enlargements $Y_n$ and $Z_n$.}
Akbulut-Yasui's corks $(W_n,f_n)$ and plugs $(W_{m,n},f_{m,n})$ in \cite{[AY1]} can give exotic enlargements by attaching 2-handles.
In this paper we consider two kinds of enlargements $Y_0=P\cup h_1$ and $Z_0=P\cup h_1\cup h_2$, where $h_1$, and $h_2$ are two 2-handles on $P$
as indicated in {\sc Figure}~\ref{PH} and the framings are both $-1$.
\begin{figure}[htbp]
\begin{center}
\includegraphics{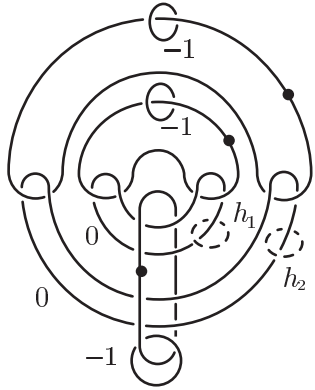}
\caption{Two kinds of attachments $Y_0=P\cup h_1$ and $Z_0=P\cup h_1\cup h_2$.}
\label{PH}
\end{center}
\end{figure}
Hence, we have
$$Y_0=\tilde{Y}\#\overline{{\Bbb C}P^2}$$
and
$$Z_0=\tilde{Z}\#^2\overline{{\Bbb C}P^2}.$$
$\tilde{Y}$ (and $\tilde{Z}$) are 4-manifolds presented by the left (and right) diagrams in {\sc Figure}~\ref{boundary}.

Let $Y_n$ and $Z_n$ define to be other enlargements obtained by twists
\begin{equation}Y_n=Y_0(P,\varphi^n)\label{defnYn}\end{equation}
and
\begin{equation}Z_n=Z_0(P,\varphi^n)\label{defnZn}\end{equation}
with respect to the embeddings $P\hookrightarrow Y_0$ and $Z_0$.
Since $Y_n$ and $Z_n$ are the 2-handle attachments of the simply-connected manifold $P$, 
they are also simply-connected and the Betti numbers $b_2$ of them are $3$ and $4$ respectively.

The g-cork $(P,\varphi^2)$ in \cite{Tan1} gives the diffeomorphisms:
\begin{equation}
Y_{n+2}\simeq Y_n\label{Yhomeo}
\end{equation}
and
\begin{equation}
Z_{n+2}\simeq Z_n.\label{Zhomeo}
\end{equation}
In this paper we use notation $\cong$ and $\simeq$ as a diffeomorphism and a homeomorphism respectively.
Hence, $Y_{2n}$ (or $Z_{2n}$) is homeomorphic to $Y_0$ (or $Z_0$) and $Y_{2n+1}$ (or $Z_{2n+1}$) is homeomorphic to $Y_1$ (or $Z_1$).
Actually $Y_n$ and $Z_n$ give the four homeomorphism types.
\begin{prop}
\label{XYZ}
Let $X$ be $Y$ or $Z$.
In $\{X_n\}$ there exist two homeomorphism types $X_0$ and $X_1$ and we have
$$X_n\simeq \begin{cases}X_0&n\equiv 0\bmod 2\\X_1&n\equiv 1\bmod 2.\end{cases}$$
\end{prop}
This proposition is proven by seeing intersection forms in later section.
The boundary $\partial Y_n$ is diffeomorphic to the 3-manifold described by the left diagram in {\sc Figure}~\ref{boundary}.
This is a $0$-surgery on $-\Sigma(2,3,5)$ as the left diagram in {\sc Figure}~\ref{boundary}.
The boundary $\partial Z_n$ is $1$-surgery of the granny knot.
The proof is in {\sc Figure}~\ref{granny}.
\begin{figure}[htbp]
\begin{center}\includegraphics{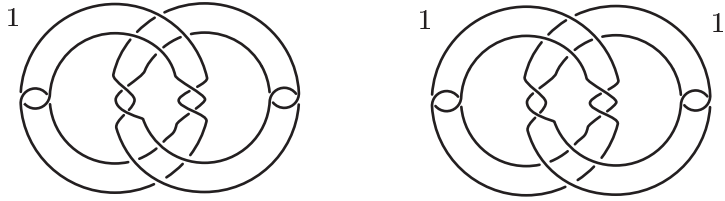}
\caption{Diagrams of $\tilde{Y}$ and $\tilde{Z}$ (as 4-manifolds) and $\partial\tilde{Y}$ and $\partial\tilde{Z}$ (as 3-manifolds).}
\label{boundary}\end{center}
\end{figure}
\begin{figure}[htpb]
\begin{center}
\includegraphics{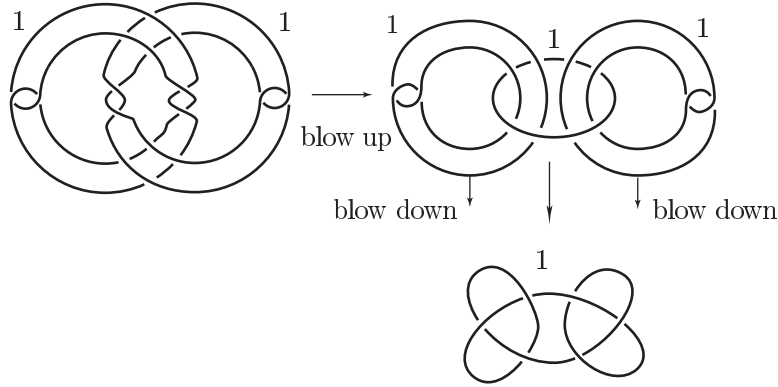}
\caption{$\partial Z_n$ is homeomorphic to $1$-surgery of the granny knot.}
\label{granny}
\end{center}
\end{figure}

From the view point of geometry, $Y_n$ and $Z_n$ have the following property.
\begin{thm}
\label{goon}
Let $n$ be a positive integer.
$Y_n$ and $Z_n$ are submanifolds of irreducible symplectic manifolds.
\end{thm}
For the differential structures, we get the following theorem.
\begin{thm}
\label{exoticex}
Let $n$ be any positive integer $n$. 
Then $Y_{2n}$ and $Y_0$ are exotic.
\end{thm}

Whether $\{Y_n\}$ are mutually non-diffeomorphic manifolds is unknown, however, we can prove the following.
\begin{thm}
\label{exoticinfinite}
Each of $\{Y_{2n}|n\in{\Bbb N} \}$ and $\{Y_{2n+1}|n\in {\Bbb N}\}$ contains infinitely many differential structures.
\end{thm}
We will prove this theorem in Section~\ref{YZ}.
The differential structures $\{Z_n\}$ satisfy the following.
\begin{thm}
\label{Z}
$\{Z_{2n}|n\ge 0\}$ and $\{Z_{2n+1}|n\ge 0\}$ are two families of mutually exotic 4-manifolds.
\end{thm}
\subsection{A twist for a rational tangle replacement.}
Let $K_i$ be a knot or link for $i=1,2$.
$K_2$ is a {\it tangle replacement} of $K_1$, if the local move $K_1\leadsto K_2$ satisfies the following:
\begin{itemize}
\item $K_2$ is a local move of $K_1$ with respect to a closed 3-ball $B^3$ that $K_i$ and $\partial B^3$ transversely intersects at $K_i\cap \partial B^3$.
\item $K_1\cap B^3$ and $K_2\cap B^3$ are proper embeddings of several arcs in $B^3$.
\item the arcs are homotopic to each other by a homotopy that fixes the boundary.
\end{itemize}
The usual crossing change of knots and links is one example of tangle replacements.
{\sc Figure}~\ref{treplace} is a picture of tangle replacement which is described schematically.

In this paper we treat the tangle replacements satisfying the following conditions.
Let $T_i$ denote $K_i\cap B^3$.
\begin{itemize}
\item $\partial T_i\subset \partial B^3$ are four points
\item $B^3\setminus K_1$ is homeomorphic to $B^3\setminus K_2$.
\end{itemize} 
The first example is the case where $B^3\setminus K_i$ is homeomorphic to the genus two handlebody.
We call the replacement {\it rational tangle replacement}.
\begin{figure}[htbp]\begin{center}\includegraphics{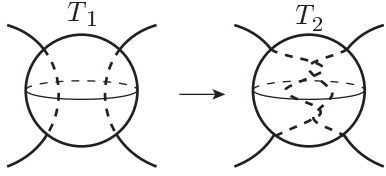}\caption{The tangle replacement $T_1\to T_2$.}\label{treplace}\end{center}\end{figure}

Let $(p,q)$ be relatively prime integers with $p$ even.
We define a diffeomorphism $\varphi_{p,q}:\partial P\to \partial P$ in Section~\ref{rationalvariation}.
The pair $(P,\varphi_{p,q})$ satisfies the following:
\begin{prop}
\label{pluggcork}
Let $p$ be an even integer with $p\neq 0$.
The twist $(P,\varphi_{p,q})$ is an infinite order
$$\begin{cases}\text{plug}&p\equiv 2\bmod 4\ \text{ or}\\\text{g-cork}&p\equiv 0\bmod 4.\end{cases}$$
\end{prop}
Let $O_n$ denote the $n$-component unlink.

\begin{thm}
\label{2bridgetheorem}
Let $X$ be a 4-manifold containing $V$ and let $K_{p,q}$ be a non-trivial 2-bridge knot.
Then there exists an embedding $i:P\hookrightarrow V\subset X$ such that the twist $(P,\varphi_{p-1,q})$ with respect to $i$
gives the knot-surgery
$$X:=X_{O_1}\leadsto X(P,\varphi_{p-1,q})=X_{K_{p,q}}.$$

Let $X_i$ be a 4-manifold containing $C$ and let $K_{p,q}$ be a non-trivial 2-bridge link.
Let $X$ be $X_1\# X_2\#S^2\times S^2$.
Then there exists an embedding $j:P\hookrightarrow X$ such that the twist $(P,\varphi_{p,q})$ with respect to $j$
gives the link-surgery
$$X=(X_1,X_2)_{O_2}\leadsto X(P,\varphi_{p,q})=(X_1,X_2)_{K_{p,q}}.$$
\end{thm}
This is a generalization of the result (Theorem~\ref{tangeprev}) that $(P,\varphi)$ is a plug and $(P,\varphi^2)$ is a g-cork.
Namely, the case of $(p,q)=(2n,1)$ corresponds to the equality $\varphi_{2n,1}=\varphi^n$.

By combining the twist and the inverse in Theorem~\ref{2bridgetheorem} we also obtain a general rational tangle replacement 
$$X_K\overset{\varphi_{p,q}^{-1}}{\leadsto} X_{O_1}\overset{\varphi_{r,s}}{\leadsto}X_{K'}.$$
\subsection{A twist for mutant knots.}
\label{mutantintro}
We call an involutive tangle replacement as in {\sc Figure}~\ref{knotm} {\it knot mutation} and we call two knots $K,K'$ which are obtained by the knot mutation {\it mutant knots}.
\begin{figure}[htbp]
\begin{center}
\includegraphics{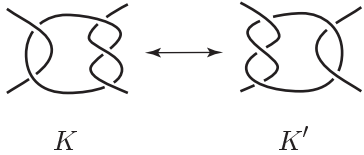}
\caption{A knot mutation.}
\label{knotm}
\end{center}
\end{figure}
It is well-known that mutant knots have similar topological properties.
Any two mutant knots have the same hyperbolic volume and HOMFLY polynomial,
in particular, the same Alexander polynomial.

The next variation of $(P,\varphi)$ is a twist between knot-surgeries for any two mutant knots.
The knot mutation is not a rational tangle replacement, because the local tangle complement is not homeomorphic to a handlebody.
Indeed, compute the fundamental group of the local tangle complement.
We found a twist $(M,\mu)$ of 4-manifold between the knot-surgeries for two mutant knots.
Let $M$ be a 4-manifold described by {\sc Figure}~\ref{m}.
A map $\mu:\partial M\to \partial M$ is defined in Section~\ref{OtherM}.
From the diagram in {\sc Figure}~\ref{m}, we can prove that $M$ is an oriented, simply-connected 4-manifold with $\partial M=\partial P\#S^2\times S^1$, $H_\ast(M)\cong H_\ast(\vee^3S^2)$, $b_3(M)=0$.
$\partial M\cong \partial P\#S^2\times S^1$ is described in {\sc Figure}~\ref{partialMM}.
\begin{figure}[htbp]
\begin{center}
\includegraphics{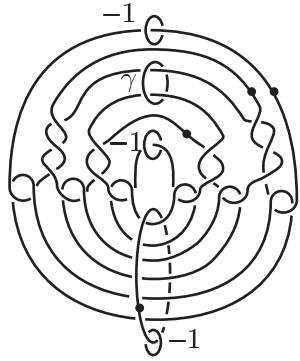}
\caption{The manifold $M$.}
\label{m}
\end{center}\end{figure}
\begin{figure}\begin{center}\includegraphics{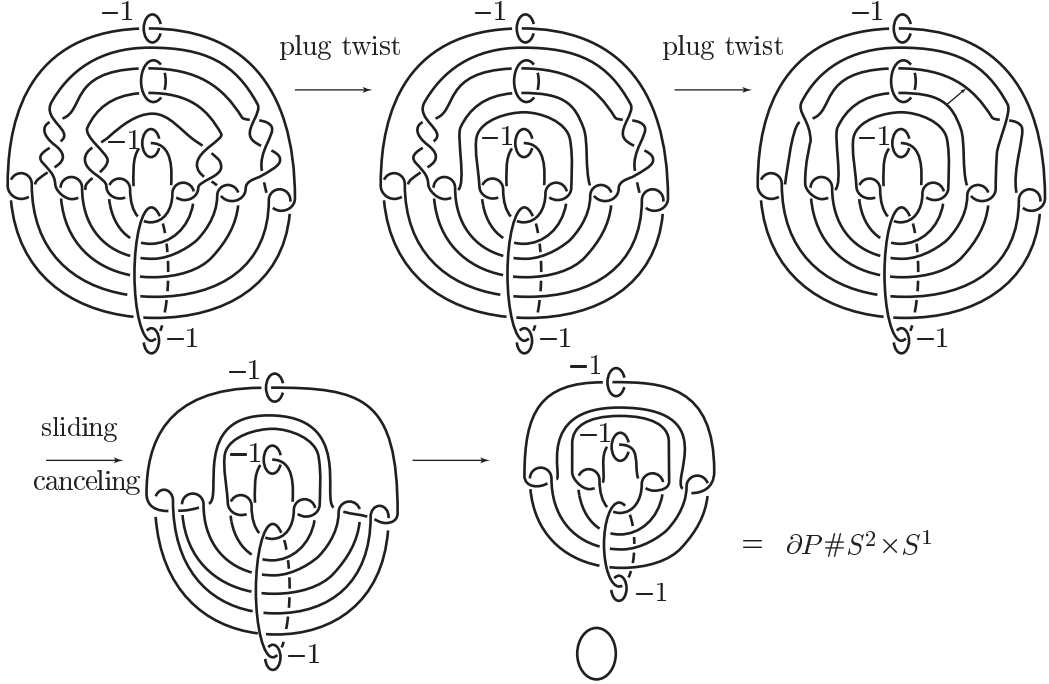}\caption{A diffeomorphism $\partial M\cong \partial P\#S^2\times S^1$}\label{partialMM}\end{center}\end{figure}

\begin{thm}
\label{Mmutant}
Let $X$ be a 4-manifold containing $V$.
Let $K,K'$ be any mutant knots.
Then there exist a twist $(M,\mu)$ and an inclusion $i:M\hookrightarrow X_K$ such that
the square of the gluing map $\mu:\partial M\to \partial M$ is homotopic to the trivial map on $\partial M$ and changes the knot-surgeries as follows:
$$X_{K'}=X_K(M,\mu).$$
\end{thm}
\begin{prop}
\label{Mmutant2}
The map $\mu:\partial M\to \partial M$ extends to a self-homeomorphism on $M$.
\end{prop}
\begin{rmk}
It are two subtle problems whether $\mu$ extends to a self-diffeomorphism on $M$.
One reason is what $M$ {\it not} be a Stein manifold.
In fact any Stein filling of a reduced 3-manifold is a boundary-sum of Stein fillings of the connected-sum components of the 3-manifold
due to \cite{[E]}.
It is well-known that the Stein filling of the $S^2\times S^1$ component must be diffeomorphic to $D^3\times S^1$ due to \cite{[Eli2]}.
$S^2\times S^1$ is a connected-sum component of $\partial M$.
These facts and $\pi_1(M)=e$ conclude that $M$ never have any Stein structure.
Therefore, if you are to interpret knot-surgeries for mutant knots as a plug or cork twist,
then you must improve the construction of $M$.

Another reason is what for mutant knots $K$ and $K'$, the Seiberg-Witten invariants are the same by Fintushel-Stern's formula in \cite{[FS]}.
\end{rmk}
\begin{rmk}
If $\mu$ can extend to $M$ as a diffeomorphism, then two knot-surgeries of all pairs of
mutant knots are diffeomorphic to each other.
Unlike the examples by Akbulut~\cite{[A1]}, and Akaho~\cite{Akh}, this diffeomorphism
suggests a meaningful map coming from knot mutation.

If $\mu$ cannot extend to inside $M$ as any diffeomorphism, then $(M,\mu)$ would be 
a not-Stein g-cork giving a subtle effect.
\end{rmk}
The twisted double $D_\mu(M):=M\cup_{\mu}(-M)$ is homeomorphic to $\#^3S^2\times S^2$.
Its diffeomorphism type is not-known.
$D_\mu(M)$ has one connected-sum component of $S^2\times S^2$, i.e. it is not irreducible.

\begin{prop}
\label{twisteddouble}
$M_0$ be a 4-manifold $M$ with a 2-handle deleted and let $\mu_0$ be
a boundary diffeomorphism $\partial M_0\to \partial M_0$ naturally induced from $\mu$.
Then we have $D_{\mu}(M)=D_{\mu_0}(M_0)\#S^2\times S^2$.
\end{prop}15756
Here we summarize several questions.
\begin{que}
Can the map $\mu$ extend to a self-diffeomorphism on $M$?
\end{que}
\begin{que}
Is $D_\mu(M)$ (or $D_{\mu_0}(M_0)$) an exotic $\#^3S^2\times S^2$ (or $\#^2S^2\times S^2$)?
\end{que}
\begin{que}
For mutant knots $K$ and $K'$, which twist $({\mathcal M},\phi)$ can realize the deformation $X_K\leadsto X_{K'}$ as a cork or plug twist?
\end{que}
\section*{Acknowledgements}
I thank Kouichi Yasui and Yuichi Yamada for giving me some useful suggestions and advice.
Specially, Yasui pointed me out that $M$ is not Stein manifold.
\section{Exotic enlargements.}
\subsection{The homeomorphism types.}
\label{YZ}
We consider four homeomorphism types $Y_0$, $Y_1$, $Z_0$ and $Z_1$ of the enlargement of $P$.
\begin{lem}
\label{intersectionform}
The intersection forms of $Y_n$ and $Z_n$ are as follows:
$$Q_{Y_n}\cong\begin{cases}\langle0\rangle\oplus H&n:\text{odd}\\\langle0\rangle\oplus\langle1\rangle\oplus \langle-1\rangle&n:\text{even,}\end{cases}$$
$$Q_{Z_n}=\begin{cases}\oplus^2H&n\text{:odd}\\\oplus^2\langle1\rangle\oplus^2 \langle -1\rangle&n\text{:even,}\end{cases}$$
where $H$ is the quadratic form presented by {\footnotesize $\begin{bmatrix}0&1\\1&0\end{bmatrix}$}.
\end{lem}
{\bf Proof of Proposition~\ref{XYZ}.}
Lemma~\ref{intersectionform}, (\ref{Yhomeo}), and (\ref{Zhomeo}) imply
the required assertion.
\hfill\qed

{\bf Proof of Lemma~\ref{intersectionform}.}
From the homeomorphisms (\ref{Yhomeo}) and (\ref{Zhomeo}), we may consider homeomorphism types $Y_0$, $Y_1$ and $Z_0$, $Z_1$ respectively.
From the picture in {\sc Figure}~\ref{PH} together with 2-handles, the intersection forms of $Y_0$ and $Z_0$ can be immediately seen $\langle 0\rangle\oplus \langle 1\rangle \oplus \langle -1\rangle$ and 
$\oplus^2\langle 1\rangle\oplus^2 \langle -1\rangle$ respectively.
The diagram of $Z_1$ is the left of {\sc Figure}~\ref{F} and the diagram of $Y_1$ is {\sc Figure}~\ref{F} with the $-2$-framed component erased.
Hence, the intersection forms of $Y_1$ and $Z_1$ are isomorphic to $\langle 0\rangle \oplus H$ and $\oplus^2H$ respectively.
\hfill\qed
\begin{figure}[htbp]
\begin{center}
\includegraphics{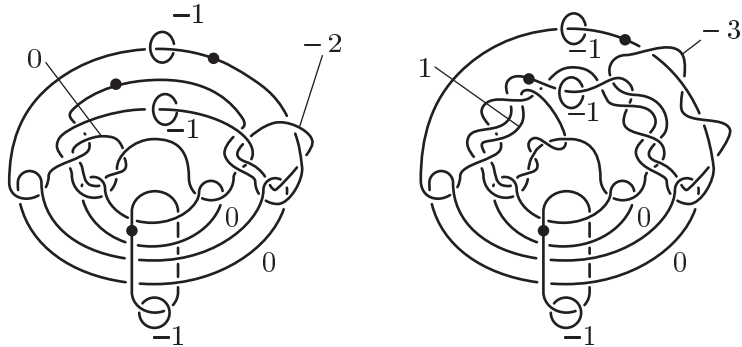}
\caption{$Z_1$ and $Z_2$.}
\label{F}
\end{center}
\end{figure}
\subsection{Infinitely many exotic structures on $Y_{0}$ and $Y_1$.}
Fintushel and Stern in \cite{[FS]} computed the Seiberg-Witten invariant of the link-surgery.
Let $L_n$ be the $(2,2n)$-torus link, in particular, the $(2,2)$-torus link is the Hopf link.
Then the Seiberg-Witten invariant is as follows:
\begin{equation}
\label{SWformula}
SW_{E(1)_{L_{n}}}=\Delta_{L_{n}}(t_1,t_2)=(t_1t_2)^{n-1}+(t_1t_2)^{n-3}+\cdots+(t_1t_2)^{-n+1}.
\end{equation}
Thus, the basic classes are the following:
\begin{equation}
\label{basicE1n}
\mathcal{B}_{E(1)_{L_n}}=\{i(t_1+t_2)|i=-n+1,n+3,\cdots,n-1\}.
\end{equation}
Each variable $t_i\in H^2(E(1)_{L_n})$ is the Poincar\'e dual $PD(2[T_i])$.
The submanifolds $T_1$, and $T_2$ are general fibers of the two copies of $E(1)$.
This implies that $E(1)_{L_n}$ are mutually non-diffeomorphic manifolds.
We prove the following lemma about $E(1)_{L_n}$:
\begin{lem}
\label{minimal}
For any positive integer $n$, $E(1)_{L_n}$ is an irreducible symplectic manifold.
\end{lem}
\begin{proof}
We assume that $E(1)_{L_n}$ has an embedded sphere $C$ with $[C]^2=-1$.
Since the intersection form is odd, $n$ is even.
We may assume $C$ is a symplectic sphere.
Let $E'$ be the blow-downed manifold along $C$.
Then the Seiberg-Witten basic classes $\mathcal{B}_{E(1)_{L_n}}$ are of form $\{k\pm PD(C)|k\in\mathcal{B}_{E'}\}$.
The basic classes $k_{\pm}=k\pm PD(C)$ satisfy $(PD(k_+)-PD(k_-))^2=4C^2=-4$.
However, from the basic classes (\ref{basicE1n}), the self-intersection number of the difference of any two of the basic classes is zero.
This is contradiction.

Since $E(1)_{L_{n}}(n\neq 0)$ is a simply-connected, minimal symplectic manifold with $b_2^+>1$, it is irreducible due to \cite{[K]}.
Thus $Z_n$ is also an irreducible symplectic manifold.
\qed\end{proof}

We prove Theorem~\ref{goon}.

{\bf Proof of Theorem~\ref{goon}.}
We will prove $Y_n\subset E(1)_{L_n}$.
The manifold $P\subset Y_0$ is embedded in $E(1)_{L_0}$ by the definition.
See \cite{[T]} for the embedding.
The twist of $(P,\varphi^n)$ via the embedding $P\hookrightarrow E(1)_{L_0}$ gets $E(1)_{L_n}$.
Then $Y_0$ changes to $Y_n$ in $E(1)_{L_n}$.
This result is due to Theorem~\ref{crossingchange} or \cite{Tan1}.
From Lemma~\ref{minimal}, $Y_n$ and $Z_n$ are submanifolds of an irreducible symplectic 4-manifold.
For $n=1$, see {\sc Figure}~\ref{FF}.

Applying the same twist for $Z_0\subset E(1)_{L_0}$, we can obtain an embedding $Z_n\hookrightarrow E(1)_{L_n}$.
\qed

Notice that each of 2-handles $h_1$ or $h_2$ in $Y_n$ and $Z_n$ corresponds to the sections in $E(1)-\nu(T^2)$.

{\bf Proof of Theorem~\ref{exoticex}.}
From Theorem~\ref{goon}, if $n$ is positive, then $Y_{2n}$ is irreducible,
however $Y_0$ has a $(-1)$-sphere.
Thus $Y_{2n}$ is not diffeomorphic to $Y_0$.
From Proposition~\ref{XYZ}, $Y_0$ and $Y_{2n}$ are homeomorphic.
\hfill\qed

\begin{figure}[htbp]
\begin{center}
\includegraphics{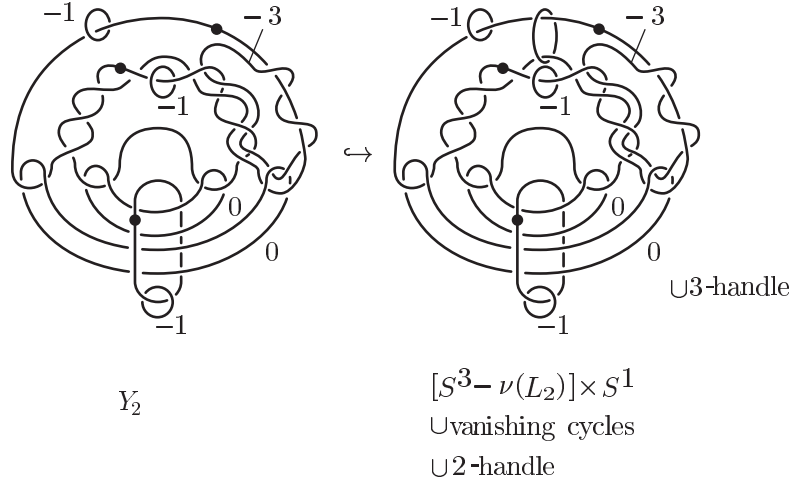}
\caption{$Y_2\hookrightarrow [S^3-\nu(L_n)]\times S^1\cup3\text{ vanishing cycles$\cup$2-handle}$.}
\label{FF}
\end{center}
\end{figure}

Next, we will prove the existence of infinitely many mutually exotic differential structures in $\{Y_n\}$.
First, we prove the following lemmas:
\begin{lem}
\label{fundamental}
Let $Q$ be a quadratic form $\langle 0\rangle\oplus\langle -1\rangle\oplus\langle 1\rangle$ on ${\Bbb Z}^3$.
Any isomorphism $({\Bbb Z}^3,Q)\to ({\Bbb Z}^3,Q)$ preserving $Q$ is presented by 
$$\begin{pmatrix}\epsilon_1&a&b\\0&\epsilon_2&0\\0&0&\epsilon_3\end{pmatrix},$$
where each $\epsilon_i$ is $\pm1$ and $a,b$ are any integers.
\end{lem}
\begin{proof}
Let $\phi$ be any isomorphism $\phi:({\Bbb Z}^3,Q)\to ({\Bbb Z}^3,Q)$ preserving the quadratic form $Q=\langle 0\rangle\oplus\langle -1\rangle\oplus\langle 1\rangle$.
For the standard generator $\{{\bf e}_i\}$ in ${\Bbb Z}^3$, we denote the images by
$\phi({\bf e}_i)=a_i{\bf e}_1+b_i{\bf e}_2+c_i{\bf e}_3$.
Since $\phi$ preserves $Q$, we have
$$\begin{cases}
-b_1^2+c_1^2=0,\ -b_2^2+c_2^2=-1, \\
-b_3^2+c_3^2=1,\ -b_1b_2+c_1c_2=0, \\
-b_1b_3+c_1c_3=0,\  -b_2b_3+c_2c_3=0.\end{cases}$$
Solving these equations, we have $c_2=0$, $b_3=0$, $b_2=\pm1$, and $c_3=\pm1$.
Furthermore, we have $b_1=c_1=0$.
Here we put $b_2=:\epsilon_2$,  and $c_3=:\epsilon_3$.
Since the map $\phi$ is an automorphism on ${\Bbb Z}^3$, we have $a_1=:\epsilon_1$, where $\epsilon_1=\pm1$.
Hence, defining as $a=a_2$ and $b=a_3$, we get the presentation matrix of $\phi$.
\qed\end{proof}
\begin{lem}
\label{fundamental2}
Let $Q$ be a quadratic form $\langle 0\rangle\oplus H$ on ${\Bbb Z}^3$.
Any isomorphism $({\Bbb Z}^3,Q)\to ({\Bbb Z}^3,Q)$ preserving $Q$ is presented by 
$$\begin{pmatrix}\epsilon_1&a&b\\0&\epsilon_2&0\\0&0&\epsilon_2\end{pmatrix},\text{ or }\begin{pmatrix}\epsilon_1&a&b\\0&0&\epsilon_2\\0&\epsilon_2&0\end{pmatrix}$$
where each $\epsilon_i$ is $\pm1$ and $a,b$ are any integers.
\end{lem}
\begin{proof}
Let $\phi$ be any isomorphism $({\Bbb Z}^3,Q)\to ({\Bbb Z}^3,Q)$.
For the standard generator $\{{\bf e}_i\}$ in ${\Bbb Z}^3$, we denote the images by $\phi({\bf e}_i)=a_i{\bf e}_1+b_i{\bf e}_2+c_i{\bf e}_3$.
Since $\phi$ preserves $Q$, we have
$$\begin{cases}
b_1c_1=0,\ b_2c_2=0,\ b_3c_3=0, \\
b_1c_2+c_1b_2=0,\ b_1c_3+c_1b_3=0,\ b_2c_3+c_2b_3=1.
\end{cases}$$
If $b_1\neq 0$, then $c_1=0$ holds from the first equation.
Then, from $c_2=-\frac{c_1b_2}{b_1}$ and $c_3=-\frac{c_1b_3}{b_1}$, we obtain $c_2=c_3=0$.
This is contradiction for the last equation.
Thus $b_1=0$ holds.
In the same way $c_1=0$ holds.

Since $b_2b_3c_2c_3=0$, we have $b_2c_3=0$ or $c_2b_3=0$.
If $b_2c_3=0$, then $c_2b_3=1$, hence, $c_2=b_3=\pm1$ and $b_2=c_3=0$ (because $b_2c_2=0$ and $c_3b_3=0$).
If $c_2b_3=0$, then $b_2c_3=1$, hence $c_2=b_3=\pm1$ and $b_3=c_2=0$.

Since the map $\phi$ is an isomorphism, we get $a_1=\pm1$.
Therefore, we get the presented matrix of $\phi$ as above.
\qed\end{proof}

Here we introduce the following result in \cite{OS}:
\begin{prop}[\cite{OS}]
\label{ostheorem}
Suppose that $\Sigma$ is a smooth, embedded, closed 2-dimensional submanifold in a smooth 4-manifold $X$ with $b_2^+(X)>1$
and for a basic class $K$ we have $\chi(\Sigma)-[\Sigma]^2-K([\Sigma])=2n<0$.
Let $\epsilon$ denote the sign of $K([\Sigma])$.
Then the cohomology class $K+2\epsilon PD([\Sigma])$ is also a basic class.
\end{prop}

\begin{lem}
Let $m$ be a positive integer.
There exists a generator $\{T_1,T_2,S_m\}$ in $H_2(P_m)$ such that $T_i\ \ (i=1,2)$ are realized by tori and the genus of the surface realizing $S_m$ is $m(m-1)$.
The presentation matrix with respect to this generator is
$$\begin{pmatrix}0&0&0\\0&0&1\\0&1&-2m^2-m-1\end{pmatrix}.$$
\end{lem}
\begin{proof}
\begin{figure}[htpb]
\begin{center}
\includegraphics{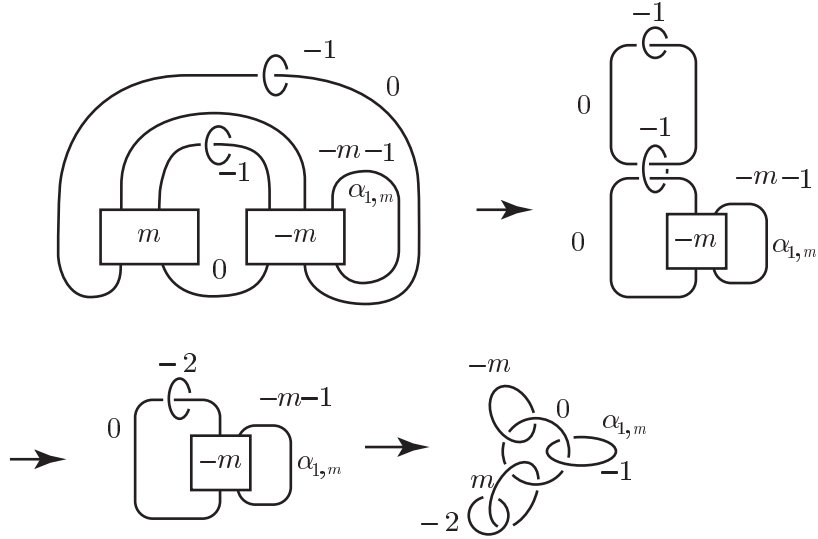}
\caption{The torus knot for $S$.}
\label{torusknotS}
\end{center}
\end{figure}
Recall that $Y_0=P\cup h_1$ and $Y_m=Y_0(P,\varphi^m)$.
Here $h_1$ is the 2-handle in {\sc Figure}~\ref{PH}.
We denote by $\alpha_1$ the attaching sphere of $h_1$
and by $\alpha_{1,m}$ the image $\varphi^m(\alpha_1)$.
The attaching sphere $\alpha_{1,m}$ is the $(m,2m+1)$-torus knot on the boundary of the $0$-handle (see the fourth pictures in {\sc Figure}~\ref{torusknotS}).

Let $m_1,m_2$ be the meridians for the link $L_m$.
Let $T_i$ be the embedded torus $T_{m_i}=m_i\times S^1$ in $Y_m$ corresponding to $m_i$.
$T_i$ can be seen in Figure~\ref{torus}.
\begin{figure}\begin{center}\includegraphics{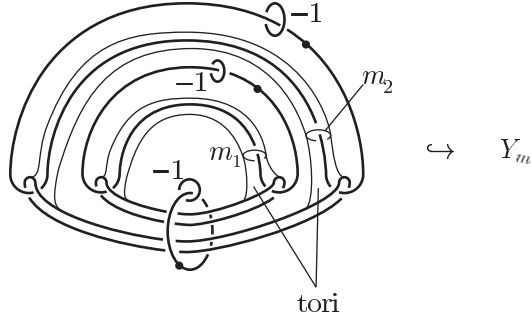}\caption{These tori are embedded in $Y_m$ as $T_1,T_2$.}\label{torus}\end{center}\end{figure}

Let $S_m$ be an embedded surface made from the union of a slice surface in $P$ of $\alpha_{1,m}$ and the core disk of $h_{1}$.
Hence, the pair $\{T_1,T_2,S_m\}$ is embedded surfaces generating $H_2(Y_m)$, because $Y_m$ consists of $\alpha_{1,m}$ and 0-framed 2-handles by canceling two 1-/2-handle canceling pairs.
The latter 0-framed 2-handles correspond to the handle decomposition of $P$.

The genus is $g(S_m)=\frac{(m-1)2m}{2}=m(m-1)$ since $\alpha_{1,m}$ is the $(m,2m+1)$-torus knot.
The self-intersection number of $S_m$ is $-2m^2-m-1$ by canceling other components by handle calculus.
The intersection of $T_2$ and $S_m$ can be understood from what attaching sphere of $T_2$ is a meridian of $S_m$ homologically in the same way as {\sc Figure}~13 in \cite{Tan1}.

Thus, the presentation matrix for the generators $\{T_1,T_2,S_m\}$ becomes the claimed one.
\end{proof}
{\bf Proof of Theorem~\ref{exoticinfinite}.}
Suppose that there exists a diffeomorphism $\delta:Y_{m}\cong Y_{n}$ for some $m,n$ with $0\le m<n$ and $n\equiv m(\bmod 2)$.
We denote by $\{T'_1,T'_2,S_n\}$ such a pair corresponding to $Y_n$.
We get a smooth inclusion:
$$S_m\subset Y_m\overset{\delta}{\to}Y_n{\hookrightarrow}E(1)_{L_n}.$$
We denote $\delta(S_m)$ simply by $S_m$ in $E(1)_{L_n}$.

Suppose that $m$ is even.
The isomorphism $f_\delta:({\Bbb Z}^3,Q_{Y_m})\to ({\Bbb Z}^3,Q_{Y_{n}})$ can be decomposed as follows:
$$({\Bbb Z}^3,Q_{Y_{m}})\to({\Bbb Z}^3,  \begin{pmatrix}0&0&0\\0&-1&0\\0&0&1\end{pmatrix})\to({\Bbb Z}^3,\begin{pmatrix}0&0&0\\0&-1&0\\0&0&1\end{pmatrix})\to({\Bbb Z}^3, Q_{Y_{n}}).$$
Using Lemma~\ref{fundamental}, we obtain the following presentation for $f_\delta$:
$$\begin{pmatrix}1&0&0\\0&-n^2-\frac{n}{2}&n^2+\frac{n}{2}+1\\0&-1&1\end{pmatrix}\begin{pmatrix}\epsilon_1&a&b\\0&\epsilon_2&0\\0&0&\epsilon_3\end{pmatrix}\begin{pmatrix}1&0&0\\0&1&-m^2-\frac{m}{2}-1\\0&1&-m^2-\frac{m}{2}\end{pmatrix}.$$
Hence, the class of $S_m$ in $Y_n$ via $\delta$ is presented as follows:
\begin{eqnarray}
[S_m]&=&\left(-(a+b)\left(m^2+\frac{m}2\right)-a\right)[T'_1]\nonumber\\
&&+\left\{(\epsilon_2-\epsilon_3)\left(m^2+\frac{m}2\right)\left(n^2+\frac{n}{2}\right)+\left(n^2+\frac{n}{2}\right)\epsilon_2\right.\nonumber\\
&&\left.-\left(m^2+\frac{m}{2}\right)\epsilon_3\right\}[T_2']+\left((\epsilon_2-\epsilon_3)\left(m^2+\frac{m}{2}\right)+\epsilon_2\right)[S_n].\label{SMcoeff}
\end{eqnarray}
Thus, we have the following intersection number
\begin{eqnarray*}
[S_m]\cdot([T_1']+[T_2'])&=&(\epsilon_2-\epsilon_3)\left(m^2+\frac{m}{2}\right)+\epsilon_2.
\end{eqnarray*}
Here putting $k=PD(\epsilon_2(n-1)([T_1']+[T_2']))$ and $\eta=\frac{1-\epsilon_2\epsilon_3}{2}$,
we have $k([S_m])=(n-1)((2m^2+m)\eta+1)> 0$.

Here we have
\begin{eqnarray*}
\chi(S_m)-[S_m]^2-k([S_m])&=&2-2m(m-1)+(2m^2+m+1)\\
&&-(n-1)((2m^2+m)\eta+1)\\
&=&3m+3-(n-1)((2m^2+m)\eta+1)\\
&=&3m-n+4-(n-1)(2m^2+m)\eta\\
&\le&3m-n+4.
\end{eqnarray*}
If $n$ satisfies $3m+4<n$, then $\chi(S_m)-[S_m]^2-k([S_m])=2\ell<0$ holds.
Using Proposition~\ref{ostheorem}, we have a basic class $k+2PD([S_m])$.

Here $S_n$ represents a section in $E(1)_{L_n}$ thus $[S_n]$ is a non-vanishing class in $H_2(E(1)_{L_n})$.
From the basic classes (\ref{basicE1n}) of $E(1)_{L_n}$, the coefficient of $[S_n]$ in $[S_m]$ must be $0$.
The coefficient of $[S_n]$ is an odd number.
See the coefficient in (\ref{SMcoeff}).
Thus, this has some contradiction.
Therefore, if $3m+4<n$ is satisfied, then $Y_n$ is not diffeomorphic to $Y_{m}$.

Suppose that $m$ is odd.
Any isomorphism $({\Bbb Z}^3,Q_{Y_m})\to ({\Bbb Z}^3,Q_{Y_{n}})$ can be decomposed as follows:
$$Q_{Y_{m}}\to  \begin{pmatrix}0&0&0\\0&0&1\\0&1&0\end{pmatrix}\to\begin{pmatrix}0&0&0\\0&0&1\\0&1&0\end{pmatrix}\to Q_{Y_{n}}.$$
Using Lemma~\ref{fundamental2}, we obtain the following presentation for $\varphi$:
$$\begin{pmatrix}1&0&0\\0&1&n^2+\frac{n+1}{2}\\0&0&1\end{pmatrix}\begin{pmatrix}\epsilon_1&a&b\\0&\epsilon_2&0\\0&0&\epsilon_2\end{pmatrix}\begin{pmatrix}1&0&0\\0&1&-m^2-\frac{m+1}{2}\\0&0&1\end{pmatrix}$$
or
$$\begin{pmatrix}1&0&0\\0&1&n^2+\frac{n+1}{2}\\0&0&1\end{pmatrix}\begin{pmatrix}\epsilon_1&a&b\\0&0&\epsilon_2\\0&\epsilon_2&0\end{pmatrix}\begin{pmatrix}1&0&0\\0&1&-m^2-\frac{m+1}{2}\\0&0&1\end{pmatrix}.$$
In fact, any automorphism preserving $\langle0\rangle\oplus H$ is the solution of
\begin{eqnarray*}
[S_m]&=&\left(b-a\left(m^2+\frac{m+1}2\right)\right)[T'_1]-\epsilon_2(m-n)\left(m+n+\frac12\right)[T_2']+\epsilon_2[S_n]
\end{eqnarray*}
or
\begin{eqnarray*}
[S_m]&=&\left(b-a\left(m^2+\frac{m+1}{2}\right)\right)[T'_1]+\epsilon_2\left(1-\left(m^2+\frac{m+1}{2}\right)\left(n^2+\frac{n+1}{2}\right)\right)[T_2']\\
&&-\epsilon_2\left(m^2+\frac{m+1}{2}\right)[S_n].
\end{eqnarray*}
Thus, we have
\begin{eqnarray*}
[S_m]\cdot([T_1']+[T_2'])=\epsilon_2\text{ or }-\epsilon_2\left(m^2+\frac{m+1}{2}\right).
\end{eqnarray*}
Here putting $k=PD(\epsilon_2(n-1)([T_1']+[T_2']))$ or $PD(-\epsilon_2(n-1)([T_1']+[T_2']))$,
we have $k([S_m])=n-1>0$ or $(n-1)\left(m^2+\frac{m+1}{2}\right)>0$ respectively.
Thus, we have
\begin{eqnarray*}
\chi(S_m)-[S_m]^2-k([S_m])&=&2-2m(m-1)+2m^2+m+1-\begin{cases}n-1\\(n-1)(m^2+\frac{m+1}{2})\end{cases}\\
&=&\begin{cases}3m+4-n\\3m+3-(n-1)(m^2+\frac{m+1}{2})\end{cases}\le 3m+4-n.
\end{eqnarray*}
If $n$ satisfies $3m+4<n$, then $\chi(S_m)-[S_m]^2-k([S_m])=2\ell<0$ holds.
Using Proposition~\ref{ostheorem}, $k+2PD([S_m])$ is also a basic class.
In the same reason as the case where $m$ is even, the coefficient of $[S_n]$ in $[S_m]$ must be $0$, namely,
we have
$$2m^2+m+1=0.$$
Since this equation does not have any integer solution, $Y_m$ is non-diffeomorphic to $Y_{n}$.

In both parities of $m$ and $n$, we can get an infinite subsequence $\{m_i\}$ in ${\Bbb N}$
such that $Y_{m_i}$ are mutually non-diffeomorphic to each other.
\qed\\
\subsection{4-manifolds $\{Z_{2n}\}$ and $\{Z_{2n+1}\}$ obtained by a g-cork $(P,\varphi^{2n})$.}
In this section we show infinitely many non-diffeomorphic exotic enlargements $Z_n$ of $P$.
$$Z_0=P\cup h_1\cup h_2=\tilde{Z}_0\#^2\overline{{\Bbb C}P^2},$$
\begin{figure}[htpb]
\begin{center}
\includegraphics{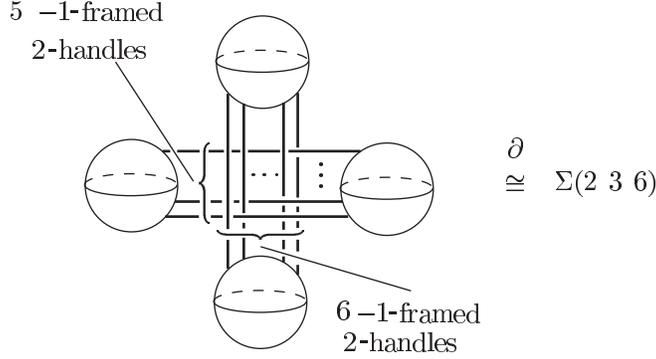}
\caption{Milnor fiber attached one 2-handle ($\tilde{M}_c(2,3,5)$). The boundary is $\Sigma(2,3,6)$.}
\label{milnor}
\end{center}
\end{figure}
{\bf Proof of Theorem~\ref{Z}.}
Let $E_{D,i}\to D^2\ (i=1,2)$ be two copies of the fibration of the complement $E(1)-\nu(T^2)$ of the neighborhood of a fiber $T^2$.
The definition of the link-surgery gives $E(1)_{L_n}=([S^3-\nu(L_n)]\times S^1)\cup_{\omega_1}E_{D,1}\cup_{\omega_2}E_{D,2}$ (see the first picture in {\sc Figure}~\ref{outline}).
Each gluing map $\omega_i$ is a map from $\partial E_{D_i}$ to one component of $\partial \nu(L_n)\times S^1$.

Here, $Dv_1$, and $Ds_1$ in $E_{D,1}$ are the neighborhoods of the compressing disk for the vanishing cycle and a section of $E_{D,1}\to D^2$.
$Dv_2, Dv_3$, and $Ds_2$ in the other component $E_{D,2}$ are the neighborhoods of the compressing disks for the vanishing cycles and a section of $E_{D,2}\to D^2$.
We use the same notation $Dv_i$, and $Ds_j$ as the parts put on $[S^3-\nu(L_n)]\times S^1$ via gluing maps $\omega_1$ and $\omega_2$ (see the second picture in {\sc Figure}~\ref{outline}).
Since the following holds:
$$E_{D,1}-Dv_1-Ds_1=M_c(2,3,6)$$
and
$$E_{D,2}-Dv_2-Dv_3-Ds_{2}=M_c(2,3,5),$$
we get
$$E(1)_{L_n}=(([S^3-\nu(L_n)]\times S^1)\cup_{i=1}^3Dv_i\cup_{i=1}^2Ds_i)\cup M_c(2,3,6)\cup M_c(2,3,5).$$
Here the Milnor fiber is defined to be the set
$$M_c(p,q,r)=\{(z_1,z_2,z_3)\in {\Bbb C}^3|z_1^p+z_2^q+z_3^r=\epsilon\text{ and }|z_1|^2+|z_2|^2+|z_3|^2\le 1\},$$
for a non-zero complex number $\epsilon$.
The handle decomposition is seen in \cite{KK}.
\begin{figure}[htbp]\begin{center}\includegraphics{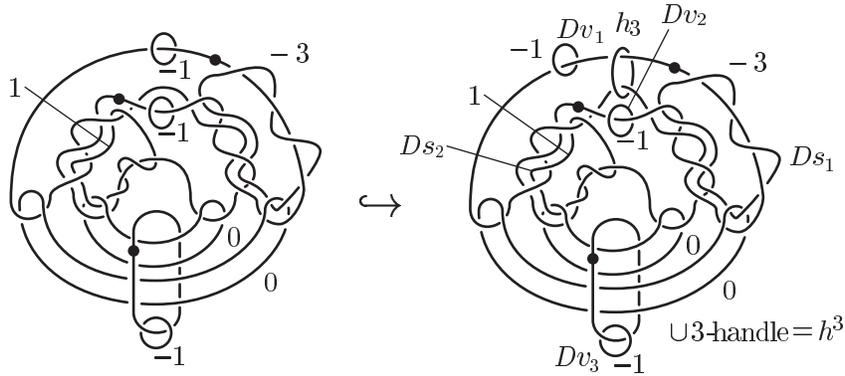}\caption{$Z_n\hookrightarrow ([S^3-\nu(L_n)]\times S^1)\cup_{i=1}^3Dv_i\cup_{i=1}^2Ds_i$}\label{ZembeddingS3Ln}\end{center}\end{figure}
\begin{figure}[htbp]\begin{center}\includegraphics{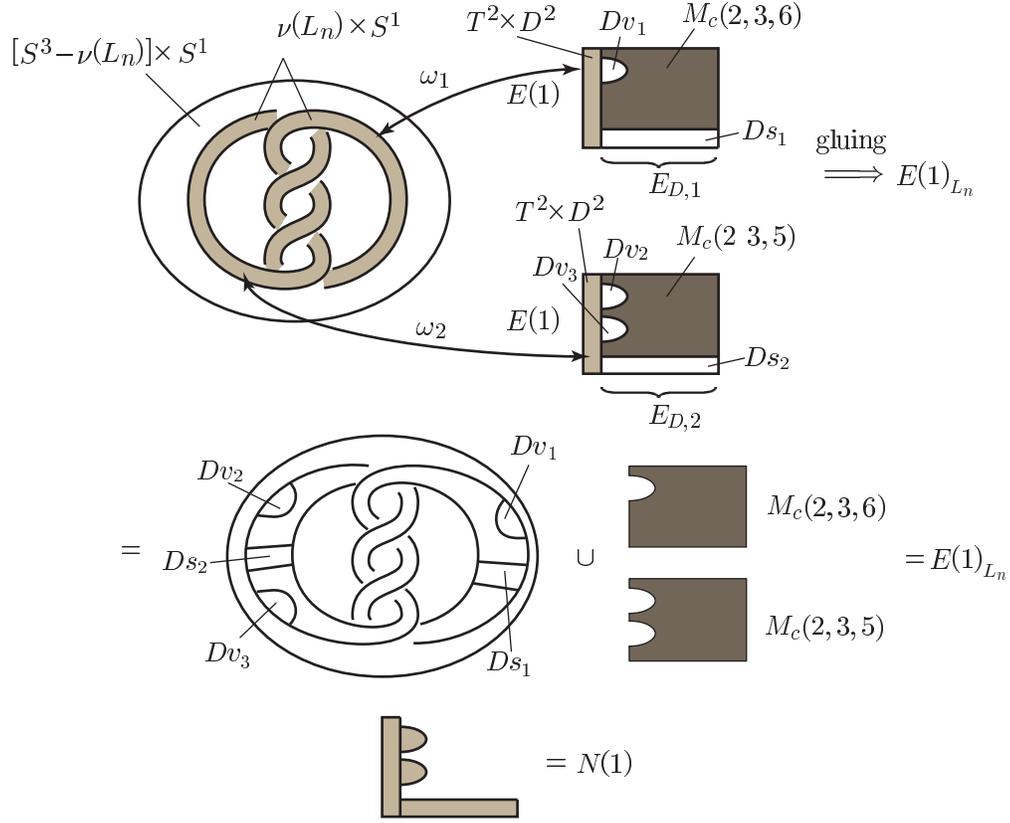}\caption{$E(1)_{L_n}=([S^3-\nu(L_n)]\times S^1)\cup E_D\cup E_D=([S^3-\nu(L_n)]\times S^1)\cup_{i=1}^3Dv_i\cup_{i=1}^2Ds_i\cup M_c(2,3,6)\cup M_c(2,3,5)$ and $N(1)$.}\label{outline}\end{center}\end{figure}
{\sc Figure}~\ref{ZembeddingS3Ln} gives $Z_n\cup h_3\cup h^3=([S^3-\nu(L_n)]\times S^1)\cup_{i=1}^3Dv_i\cup_{i=1}^2Ds_i$.
The link-surgery is constructed as follows:
$$E(1)_{L_n}=Z_n\cup h_3\cup h^3\cup M_c(2,3,6)\cup M_c(2,3,5).$$
The handles $h_3$ and $h^3$ are the 2- and 3-handle indicated in {\sc Figure}~\ref{ZembeddingS3Ln}.

Let $R$ denote the union $h_3\cup h^3$.
The attaching region of $R$ is a thickened torus $T^2\times D^1$ on $\partial Z_n$.
The boundary $\partial (Z_n\cup R)$ is the disjoint union of $\Sigma(2,3,5)$ and $\Sigma(2,3,6)$.
The isotopy class of the essential torus in $\partial Z_n$ is uniquely determined from JSJ-theory.
Thus the self-diffeomorphism on $Z_n$ can extend to $Z_n\cup R$ uniquely.

Next we attach the Milnor fibers on the boundaries $\Sigma(2,3,5)$ and $\Sigma(2,3,6)$.
Here we claim the following lemma:
\begin{lem}[\cite{[GS]},\cite{[T]}]
Any diffeomorphism on $\Sigma(2,3,5)$ or $\Sigma(2,3,6)$ extends to
$M_c(2,3,5)$ or $M_c(2,3,6)$ respectively.
\end{lem}
\proof
The proof is the same as Lemma 3.7 in \cite{[G]}.
We remark the case of $\Sigma(2,3,6)$ here.
By the result in \cite{[BO]} the diffeotopy type of $\Sigma(2,3,6)$ is isomorphic to ${\Bbb Z}/2{\Bbb Z}$.
The non-trivial diffeomorphism on $\Sigma(2,3,6)$ is the restriction of rotating by the $180^\circ$ about the horizontal line in {\sc Figure}~\ref{milnor}.
Thus the diffeomorphism extends to $M_c(2,3,6)$.
\qed

Thus, any self-diffeomorphism on $\Sigma(2,3,5)$ and $\Sigma(2,3,6)$ can extend to $M_c(2,3,5)$ or $M_c(2,3,6)$.

The diffeomorphism on $Z_n$ can extend to $Z_{n}\cup R\cup M_c(2,3,5)\cup M_c(2,3,6)=E(1)_{L_n}.$
This means that the diffeomorphism type of $E(1)_{L_n}$ is determined by that of $Z_n$.
Conversely, if $m\neq n$, then $Z_n$ and $Z_m$ are non-diffeomorphic.
\qed\\
Hence, we have the following corollary.
\begin{cor}
Any diffeomorphism $Z_n\to Z_m$ extends to a diffeomorphism $E(1)_{L_n}\to E(1)_{L_m}$.
\end{cor}
Hence, in this case $Z$ has the same role as Gompf's nuclei $N$ in \cite{[G]}.
\section{Some variations of plug twists.}
\label{rationalvariation}
In this section, combining the plug twist $(P,\varphi)$ and other twists, we show the 2-bridge knot-surgery and 2-bridge link-surgery (Theorem~\ref{2bridgetheorem})
are produced by the same $P$.
\subsection{The 2-bridge knot-surgery}
For an irreducible fraction $p/q$, take the continued fraction 
$$p/q=a_1-\frac{1}{a_2-\frac{1}{a_3-\cdots-\frac{1}{a_n}}}=[a_1,a_2,a_3,\cdots,a_n].$$

The continued fraction determines the 2-bridge knot or link diagram as {\sc Figure}~\ref{2-bridge},
where \fbox{$k$} in the figure stands for the $k$-half twist.
The isotopy type of $K_{p,q}$ depends only on the relatively prime integers $(p,q)$ and does not depend on the way of the continued fraction.
\begin{figure}[htpb]
\begin{center}
\includegraphics{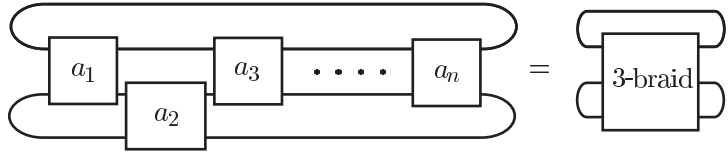}
\caption{An example of the 2-bridge knot or link $K_{p,q}$.}
\label{2-bridge}
\end{center}
\end{figure}

The following deformations of coefficients do not change the isotopy class of $K_{p,q}$ and the rational number $p/q$:
\begin{equation}
(a_1,\cdots,a_i,a_{i+1},\cdots,a_n)\leftrightarrow (a_1,\cdots,a_i\pm1,\pm1,a_{i+1}\pm1,\cdots,a_n)\label{defom1}
\end{equation}
\begin{equation}
(a_1,\cdots,a_n)\leftrightarrow (\pm1,a_1\pm1,\cdots,a_n), (a_1,\cdots,a_n\pm1,\pm1)\label{defom}
\end{equation}
By using this deformation, for any irreducible fraction $p/q$ we get the continued fraction
$$p/q=[b_1,b_2,\cdots,b_N]$$
such that $N$ is an odd number and $b_3,b_5,\cdots,b_{N}$ are all even.
If $b_1$ is odd or even, then $K_{p,q}$ is a knot or 2-component link respectively.
We define the 3-braid indicating as in the right of {\sc Figure}~\ref{2-bridge} with respect to $(b_1,b_2,\cdots,b_N)$ to be $B_{p,q}$.

Let $p$ be an even integer.
Then we take a continued fraction $p/q=[b_1,\cdots,b_N]$ as above.
\begin{figure}[htbp]
\begin{center}
\includegraphics{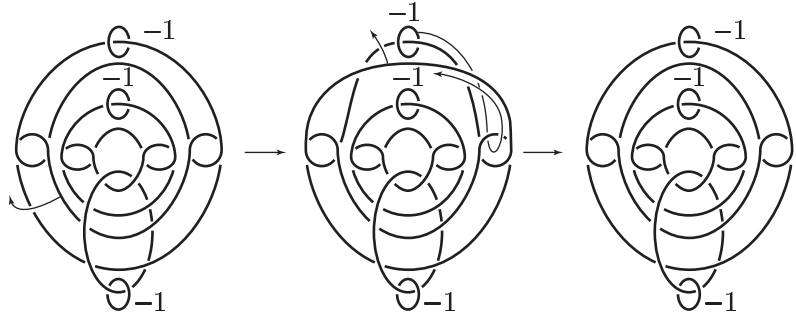}
\caption{The definition of $\psi$.}
\label{twist}
\end{center}
\end{figure}
Namely, $N$ is an odd number and $b_1,b_3,\cdots,b_N$ are even integers.
We denote the map $\psi:\partial P\to\partial P$ as in {\sc Figure}~\ref{twist}.
We define $\varphi_{p,q}:\partial P\to \partial P$ as follows:
\begin{equation}
\label{varphipq}
\varphi_{p,q}:=\varphi^{\frac{b_N}{2}}\circ\psi^{b_{N-1}}\circ\cdots\circ \varphi^{\frac{b_3}{2}}\circ\psi^{b_2}\circ \varphi^{\frac{b_1}{2}}.
\end{equation}
This definition may depend on the way of continued fraction of $p/q$.
We choose such a continued fraction for the fraction $p/q$.
Here we prove Proposition~\ref{pluggcork}.

\begin{proof}
First, we show that $\varphi_{p,q}$ is not a torsion element.
Let $B_3$ be the 3-braid group with the following presentation:
$$B_3=\langle \sigma_1,\sigma_2|\sigma_1\sigma_2\sigma_1=\sigma_2\sigma_1\sigma_2\rangle,$$
and let $B_3^0$ be a subgroup generated by $\sigma_1$ and $\sigma_2^2$.
The generators $\sigma_1$ and $\sigma_2$ are as in {\sc Figure}~\ref{gen}.
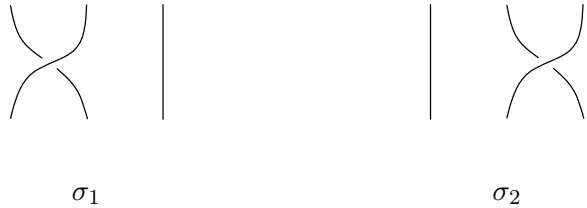
\begin{figure}[thpb]\begin{center}
{\unitlength 0.1in%
\begin{picture}( 30.0000,  9.3500)( 14.0200,-15.3500)%
%
\special{pn 8}%
\special{pa 3600 600}%
\special{pa 3600 1200}%
\special{fp}%
%
\special{pn 8}%
\special{pa 2200 600}%
\special{pa 2200 1200}%
\special{fp}%
\put(40.0000,-16.0000){\makebox(0,0){$\sigma_2$}}%
\put(18.0000,-16.0000){\makebox(0,0){$\sigma_1$}}%
%
\special{pn 8}%
\special{pa 1800 602}%
\special{pa 1798 638}%
\special{pa 1796 673}%
\special{pa 1792 707}%
\special{pa 1786 739}%
\special{pa 1778 770}%
\special{pa 1766 798}%
\special{pa 1750 822}%
\special{pa 1730 844}%
\special{pa 1705 862}%
\special{pa 1678 877}%
\special{pa 1648 891}%
\special{pa 1617 904}%
\special{pa 1587 918}%
\special{pa 1557 933}%
\special{pa 1529 949}%
\special{pa 1504 969}%
\special{pa 1484 991}%
\special{pa 1466 1016}%
\special{pa 1451 1044}%
\special{pa 1438 1074}%
\special{pa 1427 1105}%
\special{pa 1417 1137}%
\special{pa 1408 1171}%
\special{pa 1402 1196}%
\special{fp}%
%
\special{pn 8}%
\special{pa 1402 606}%
\special{pa 1408 639}%
\special{pa 1415 671}%
\special{pa 1424 703}%
\special{pa 1434 733}%
\special{pa 1447 762}%
\special{pa 1463 788}%
\special{pa 1483 812}%
\special{pa 1506 834}%
\special{pa 1531 854}%
\special{pa 1558 874}%
\special{pa 1566 880}%
\special{fp}%
%
\special{pn 8}%
\special{pa 1646 936}%
\special{pa 1694 980}%
\special{pa 1716 1003}%
\special{pa 1736 1028}%
\special{pa 1753 1054}%
\special{pa 1767 1082}%
\special{pa 1778 1112}%
\special{pa 1798 1174}%
\special{pa 1804 1198}%
\special{fp}%
%
\special{pn 8}%
\special{pa 4398 600}%
\special{pa 4396 636}%
\special{pa 4394 671}%
\special{pa 4390 705}%
\special{pa 4384 737}%
\special{pa 4376 768}%
\special{pa 4364 796}%
\special{pa 4348 820}%
\special{pa 4328 842}%
\special{pa 4303 860}%
\special{pa 4276 875}%
\special{pa 4246 889}%
\special{pa 4215 902}%
\special{pa 4185 916}%
\special{pa 4155 931}%
\special{pa 4127 947}%
\special{pa 4102 967}%
\special{pa 4082 989}%
\special{pa 4064 1014}%
\special{pa 4049 1042}%
\special{pa 4036 1072}%
\special{pa 4025 1103}%
\special{pa 4015 1135}%
\special{pa 4006 1169}%
\special{pa 4000 1194}%
\special{fp}%
%
\special{pn 8}%
\special{pa 4000 604}%
\special{pa 4006 637}%
\special{pa 4013 669}%
\special{pa 4022 701}%
\special{pa 4032 731}%
\special{pa 4045 760}%
\special{pa 4061 786}%
\special{pa 4081 810}%
\special{pa 4104 832}%
\special{pa 4129 852}%
\special{pa 4156 872}%
\special{pa 4164 878}%
\special{fp}%
%
\special{pn 8}%
\special{pa 4244 934}%
\special{pa 4292 978}%
\special{pa 4314 1001}%
\special{pa 4334 1026}%
\special{pa 4351 1052}%
\special{pa 4365 1080}%
\special{pa 4376 1110}%
\special{pa 4396 1172}%
\special{pa 4402 1196}%
\special{fp}%
\end{picture}}%
\caption{The generators $\sigma_1,\sigma_2$ in $B_3$.}\label{gen}\end{center}\end{figure}
This group is a normal subgroup in $B_3$ and gives a homomorphism $\pi:B_3^0\to MCG(\partial P)$ defined to be a map satisfying $\pi(\sigma_1)=\psi$ and $\pi(\sigma_2^2)=\varphi$.
Hence, $\varphi_{p,q}$ lies in $\pi(B_3^0)$.
Here $MCG(\partial P)$ is the mapping class group of $\partial P$.

\begin{clm}
$B_3^0\cong F_2\rtimes {\Bbb Z}$,
where $F_2$ is the rank $2$ free group.
\end{clm}
\begin{proof}
We have the following short exact sequence:
$$1\to F_2\overset{f_1}{\to}B_3^0\overset{f_2}{\to} {\Bbb Z}\to 0,$$
where $f_2$ is the number of half-twists between the first string and the second string, namely, it is the map $B_3^0\to B_2\cong {\Bbb Z}$ obtained by forgetting the third string.
The subgroup in $B_3^0$ satisfying $f_2=0$ is considered as the homotopy class of a path on the 2 holed disk with a base point.
Thus we have $\text{Ker}(f_2)\cong F_2$.
This exact sequence is splittable since $B_2\cong\langle\sigma_1\rangle$ is the subgroup in $B_3^0$ as a lift of $f_2$. \qed
\end{proof}
Since $F_2$ and ${\Bbb Z}$ are torsion-free, $F_2\rtimes {\Bbb Z}$ is also torsion-free.
This means that if $\varphi_{p,q}$ is torsion, then $\varphi_{p,q}=\text{id}$ holds.
Since the twist $(P,\varphi_{p,q})$ of $E(1)_{O_2}=3{\Bbb C}P^2\#19\overline{{\Bbb C}P^2}$ is
trivial, namely, $\Delta_{K_{p,q}}(t_1,t_2)=0$.
The 2-bridge knot with Alexander polynomial zero is the 2-component unlink only.
Therefore if $p\neq 0$, then $\varphi_{p,q}$ is not torsion.

We compute the intersection form of $D_{\varphi_{p,q}}(P)$.
The double is described in {\sc Figure}~\ref{tanglae} (the case of $N=3$).
The two ($0$-framed) fine curves are the attaching spheres of the upper manifolds of the double.
The curve is parallel to the thick curve in each box with $\pm b_{2k+1}$-half twist and is twisted in each box with $\pm b_{2k}$-half twist.
The parallel and twisted diagram is described in {\sc Figure}~\ref{insidbox}.
The first deformation (homeomorphism) in {\sc Figure}~\ref{insidbox} is also seen in \cite{Tan1} and the second and fourth deformations (diffeomorphisms)
are also seen in \cite{Tan1}.
The third deformation in {\sc Figure}~\ref{insidbox} is an isotopy of the diagram.
Hence, the intersection form is $\oplus^2\begin{pmatrix}0&1\\1&-\frac{1}{2}(b_1+b_3+\cdots+b_N)\end{pmatrix}$.

We claim the following:
\begin{lem}
Let $[b_1,\cdots,b_N]$ be a continued fraction of $p/q$ with $N$ an odd natural number.
If $b_1,b_3,\cdots b_N$ are all even integers, then $p\equiv (-1)^{\frac{N-1}{2}}(b_1+b_3+\cdots+b_N)\bmod 4$.
\end{lem}
\begin{proof}
The integer $p$ is equal to the (1,1)-component in the following matrix.
$$\begin{pmatrix}b_1&-1\\1&0\end{pmatrix}\begin{pmatrix}b_2&-1\\1&0\end{pmatrix}\cdots \begin{pmatrix}b_N&-1\\1&0\end{pmatrix}.$$
Since we have
$$\begin{pmatrix}b_1&-1\\1&0\end{pmatrix}\begin{pmatrix}b_2&-1\\1&0\end{pmatrix}\begin{pmatrix}b_3&-1\\1&0\end{pmatrix}\equiv \begin{pmatrix}-b_1-b_3&1-b_1b_2\\1-b_2b_3&-b_2\end{pmatrix}\bmod 4.$$
Suppose that 
\begin{equation}\prod_{l=1}^{2k+1}\begin{pmatrix}b_l&-1\\1&0\end{pmatrix}\equiv(-1)^{k}\begin{pmatrix}\sum_{l=0}^kb_{2l+1}&-1+\sum_{s=1}^kc_sb_{2s-1}\\1+\sum_{s=1}^kd_sb_{2s+1}&e\end{pmatrix}\bmod 4,\label{induction}\end{equation}
where $c_i,d_i,e$ are some integers.
Then we have
\begin{eqnarray*}
&&\begin{pmatrix}\sum_{s=0}^kb_{2s+1}&-1+\sum_{s=1}^kc_sb_{2s-1}\\1+\sum_{s=1}^kd_sb_{2s+1}&e\end{pmatrix}\begin{pmatrix}b_{2k+2}&-1\\1&0\end{pmatrix}\begin{pmatrix}b_{2k+3}&-1\\1&0\end{pmatrix}\\
&=&\begin{pmatrix}\sum_{s=0}^kb_{2s+1}&-1+\sum_{s=1}^kc_sb_{2s-1}\\1+\sum_{s=1}^kd_sb_{2s+1}&e\end{pmatrix}\begin{pmatrix}b_{2k+2}b_{2k+3}-1&-b_{2k+2}\\b_{2k+3}&-1\end{pmatrix}\\
&\equiv&\begin{pmatrix}-\sum_{s=0}^{k+1}b_{2s+1}&b_{2k+2}\sum_{s=0}^kb_{2s+1}+1+\sum_{s=1}^kc_sb_{2s-1}\\-1-b_{2k+2}b_{2k+3}-\sum_{s=1}^kd_sb_{2s+1}+eb_{2k+3}&-e'\end{pmatrix}\bmod 4\\
&=&-\begin{pmatrix}\sum_{s=0}^{k+1}b_{2s+1}&-1+\sum_{s=1}^{k+1}c'_sb_{2s-1}\\1+\sum_{s=1}^{k+1}d'_sb_{2s+1}&e'\end{pmatrix},
\end{eqnarray*}
where $c_i',d_i',e'$ are some integers.
Thus (\ref{induction}) holds for $k+1$ instead of $k$.
The induction implies $p\equiv (-1)^{\frac{N-1}{2}}(b_1+b_3+\cdots+b_N)\bmod 4$.
\qed

We go back to the proof of Proposition~\ref{pluggcork}.
The intersection form of $D_{\varphi_{p,q}}(P)$ is 
$$\oplus^2\begin{pmatrix}0&1\\1&(-1)^{\frac{N+1}{2}}\frac{p}{2}\end{pmatrix}\cong \begin{cases}\oplus^2\langle 1\rangle\oplus^2\langle -1\rangle &p\equiv 2\bmod 4\\\oplus^2H&p\equiv 0\bmod 4\end{cases}$$
The Boyer's result means that if $p\equiv 2\bmod 4$, then $(P,\varphi_{p,q})$ is a plug and if $p\equiv 0\bmod 4$, then $(P,\varphi_{p,q})$ is a g-cork.
\qed\end{proof}
\begin{figure}[htpb]
\begin{center}\includegraphics{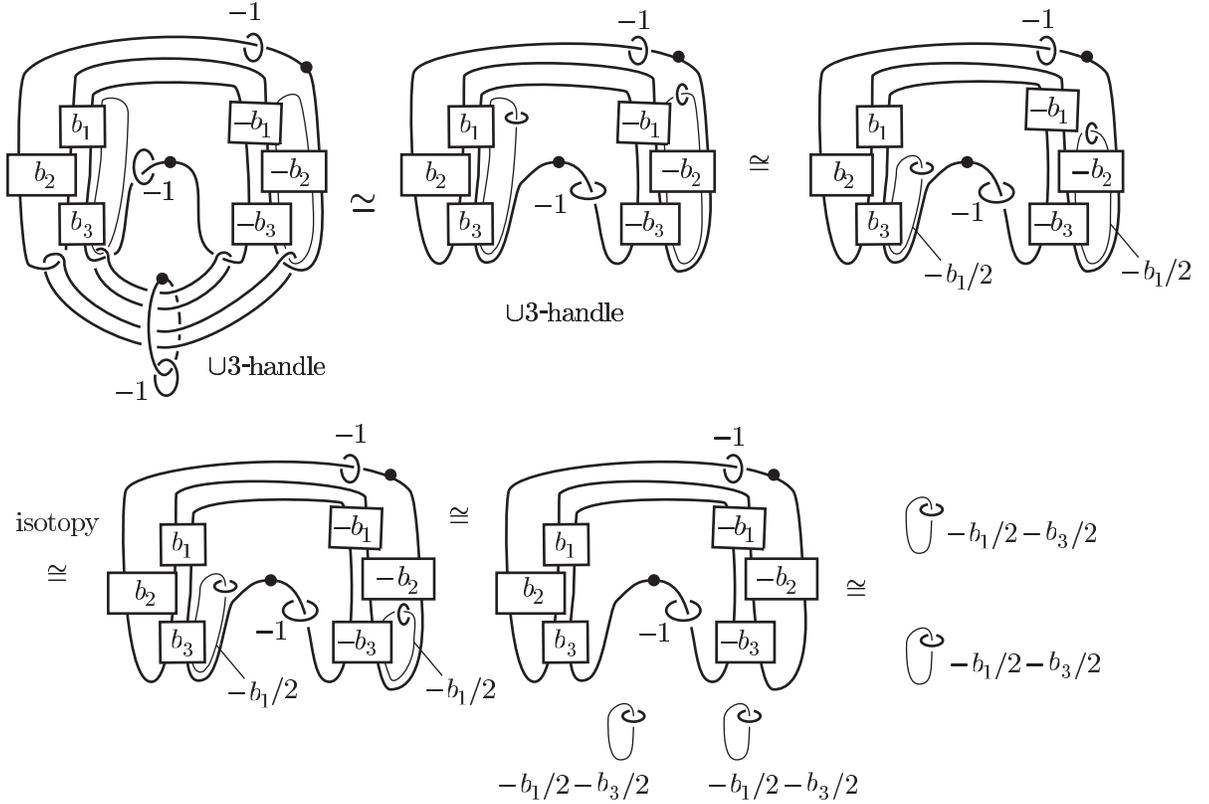}\caption{The homeomorphism type of $D_{\varphi_{p,q}}(P)$.}\label{tanglae}\end{center}
\end{figure}
\begin{figure}[htpb]
\begin{center}\includegraphics{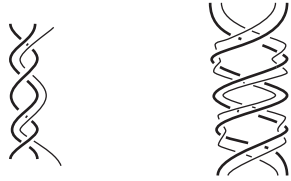}\caption{The local pictures of the fine curves in the box \fbox{$\pm b_{2k-1}$} and \fbox{$\pm b_{2k}$}
in the first picture in {\sc Figure}~\ref{tanglae}.
(the cases of $b_{2k-1}=4$ or $b_{2k}=4$.}\label{insidbox}\end{center}
\end{figure}
\end{proof}
We decompose Theorem~\ref{2bridgetheorem} into two propositions (Proposition~\ref{knot1} and \ref{link1}).
\begin{prop}
\label{knot1}
Let $X$ be a 4-manifold containing $V$ and $K_{p,q}$ be a non-trivial 2-bridge knot (i.e. $p$ is an odd number).
Then there exists an embedding $i:P\hookrightarrow V\subset X$ such that the twist $(P,\varphi_{p-1,q})$ gives the deformation:
$$X_{K_{p,q}}=X(P,\varphi_{p-1,q},i),$$
where the embedding $i$ is defined in {\sc Figure}~\ref{twist3} and independent of $K_{p,q}$.
\end{prop}
\begin{proof}
\begin{figure}[htbp]
\begin{center}
\includegraphics{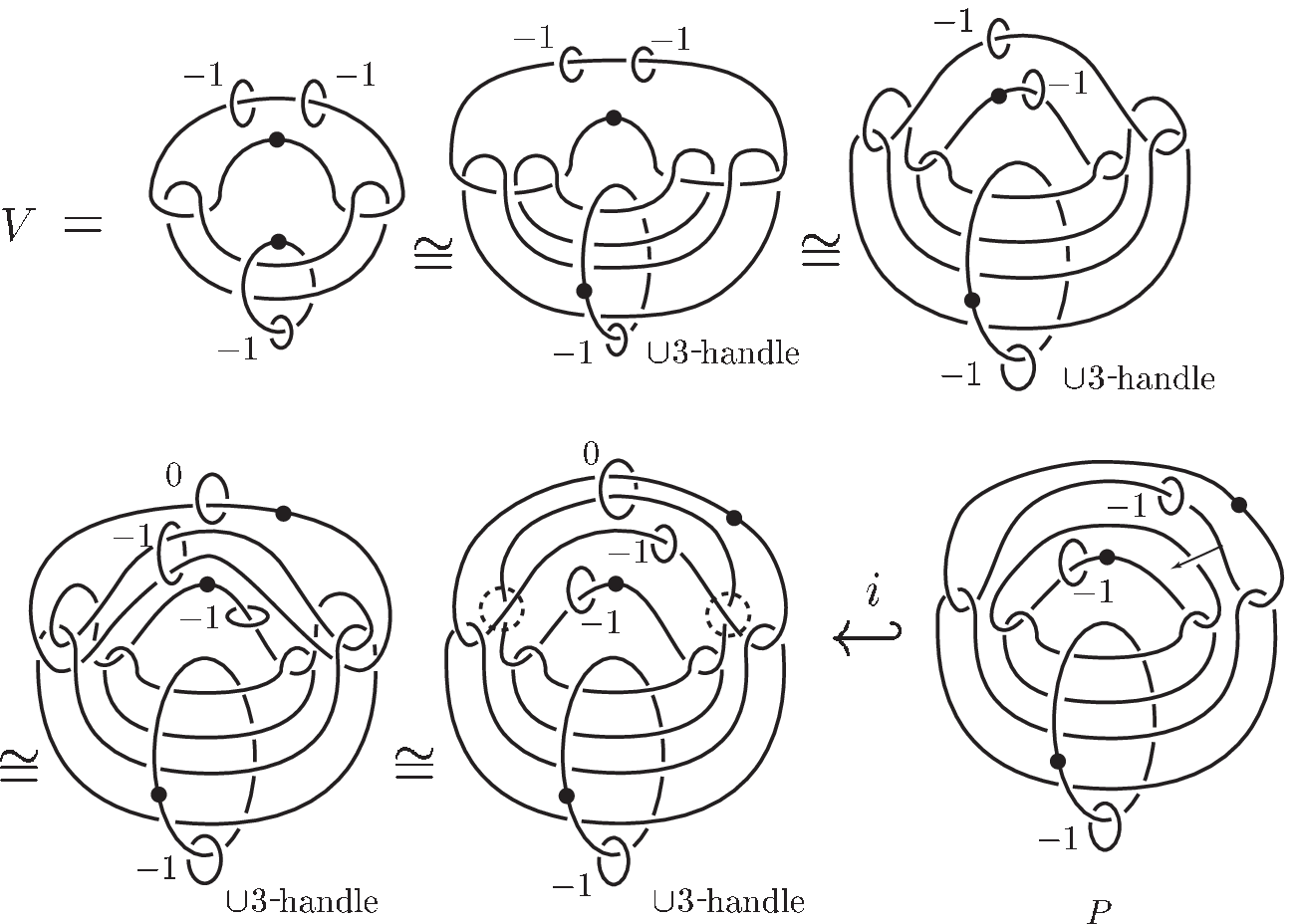}
\caption{The embedding $i:P\hookrightarrow V$.}
\label{twist3}
\end{center}
\end{figure}
The embedding $i:P\hookrightarrow V$ is constructed in {\sc Figure}~\ref{twist3}.
The twist $(P,\varphi^{\frac{b_1-1}{2}})$ is described in the first deformation in {\sc Figure}~\ref{plugb1}.
Consecutively, we do the twist $(P,\psi^{b_2})$ (the second deformation in {\sc Figure}~\ref{plugb1}).
Continuing the twists along (\ref{varphipq}), we totally obtain the twist $(P,\varphi_{p,q})$
in the last picture in {\sc Figure}~\ref{plugb1}.
Here $-B_{p,q}$ is the mirror image of the braid $B_{p,q}$.
\begin{figure}[htbp]
\begin{center}
\includegraphics{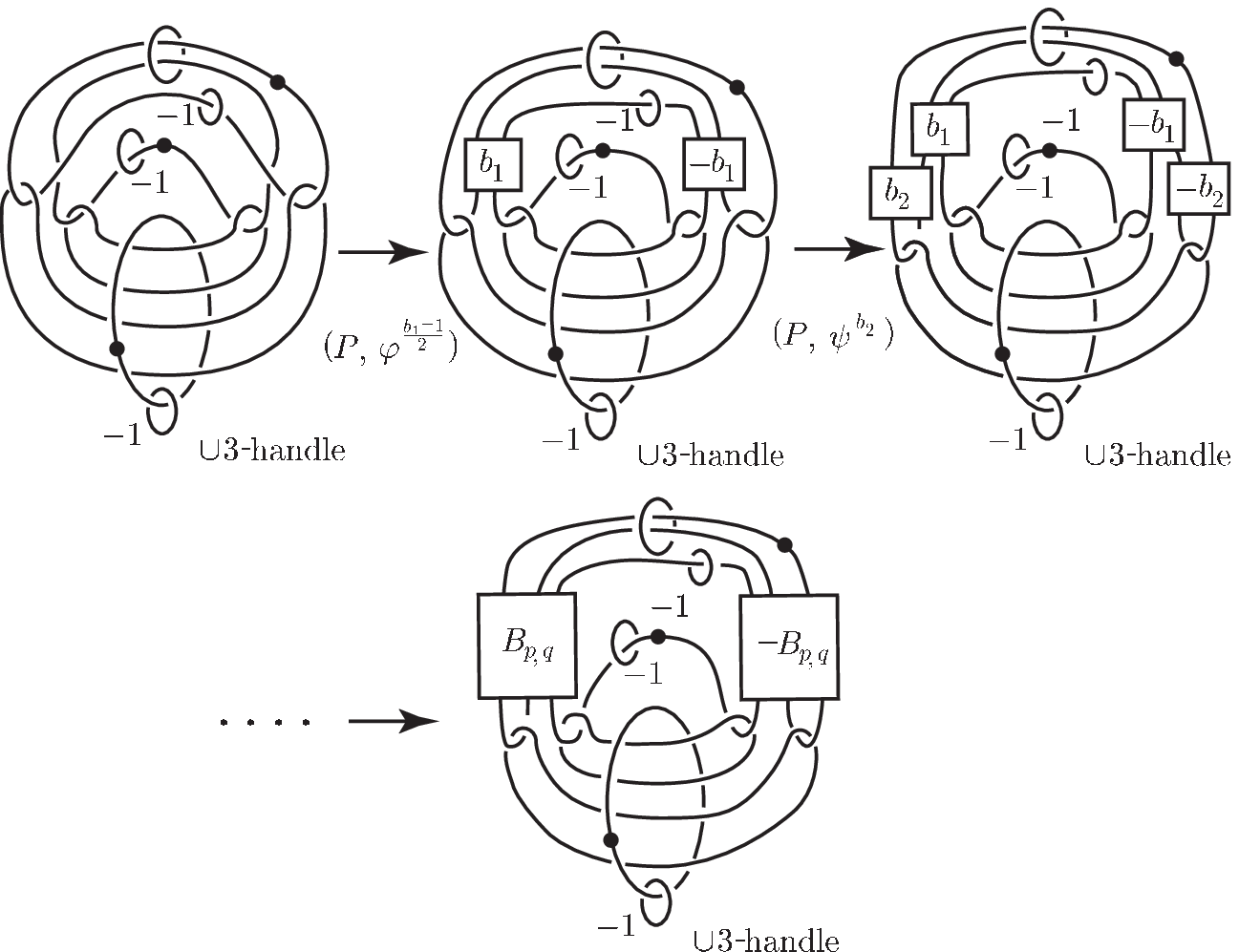}
\caption{The construction of the twist by $(P,\varphi_{p-1,q})$ of $V$. The box \fbox{$n$} stands for the $n$-half twist.}
\label{plugb1}
\end{center}
\end{figure}

We compute the intersection form of the twisted double $D_{\varphi_{p,q}}(P)$.
\hfill\qed
\end{proof}
\begin{rmk}
In the similar way, we can also construct another embedding $i':P\hookrightarrow V$ by changing the crossings in the broken circles in {\sc Figure}~\ref{twist3}.
This embedding is different from $i$, because the twist $(P,\varphi_{p-1,q})$ gives $V_{K_{p-2,q}}$.
In general, the Alexander polynomials of $K_{p-2,q}$ and $K_{p,q}$ are different.
\end{rmk}
\subsection{2-bridge link-surgery}
We consider the case of link-surgery.
Let $C$ be a cusp neighborhood (i.e., Kodaira's singular fibration II).
The handle decomposition of $C$ is described in {\sc Figure}~\ref{KodIII}.
We denote $C\# C\#S^2\times S^2$ by $W$.
\begin{prop}
\label{link1}
Let $X_i$ ($i=1,2$) be two 4-manifolds containing $C$ and 
let $X$ be $X_1\#X_2\#S^2\times S^2$.
If $K_{p,q}$ is a 2-bridge link (i.e. $p$ is an even number), then
there exists an embedding $j:P\hookrightarrow W\subset X$ such that the twist $(P,\varphi_{p,q})$ gets
$$X(P,\varphi_{p,q})=(X_1,X_2)_{K_{p,q}},$$
where the embedding $j:P\hookrightarrow X:=X_1\#X_2\#(S^2\times S^2)$
is the one obtained by the same way as indicated in {\sc Figure}~\ref{plugb1}.
\end{prop}

\begin{proof}
The application of $\varphi_{p,q}$ to {\sc Figure}~\ref{linksurgery} 
in the same way as {\sc Figure}~\ref{plugb1}
gives the twist $X(P,\varphi_{p,q})=(X_1,X_2)_{K_{p,q}}$.
\qed
\end{proof}
\begin{figure}[htbp]
\begin{center}
\includegraphics{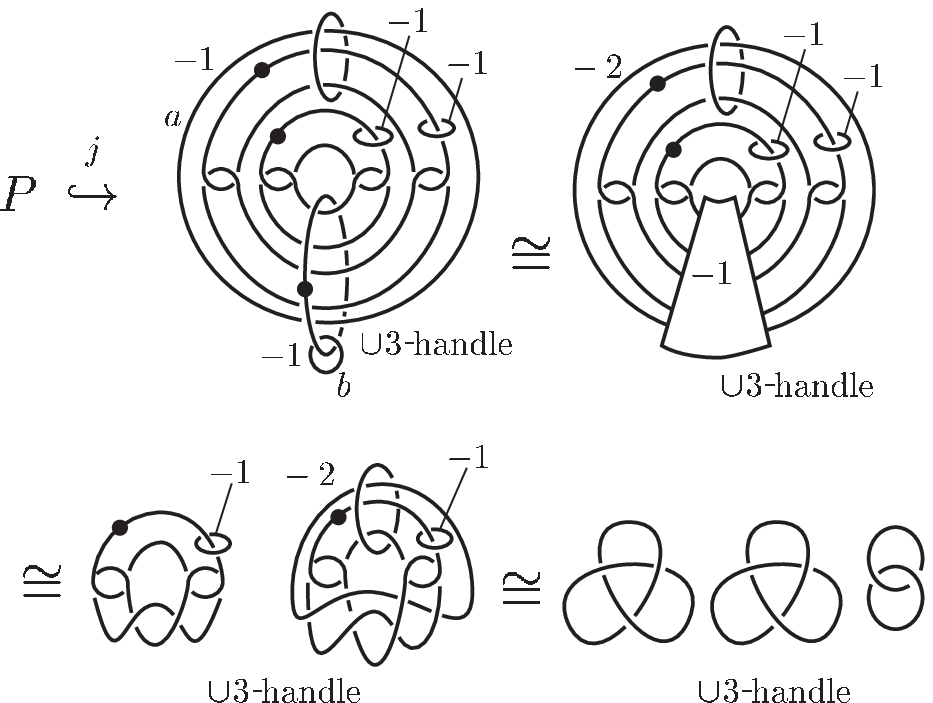}
\caption{The embedding $j:P\hookrightarrow C\#C\#S^2\times S^2=:W$.}
\label{linksurgery}
\end{center}
\end{figure}
{\bf Proof of Theorem~\ref{2bridgetheorem}.}
Let $K$ be a 2-bridge knot or link.
Then Proposition~\ref{knot1} and \ref{link1}, it follows the required assertion.
\qed 
\subsection{A twist $(M,\mu)$.}
\label{OtherM}
Let $M$ be the manifold described in {\sc Figure}~\ref{m}.
We factorize the knot mutation into the three processes as in {\sc Figure}~\ref{mufact}.
\begin{figure}[htbp]
\begin{center}
\includegraphics{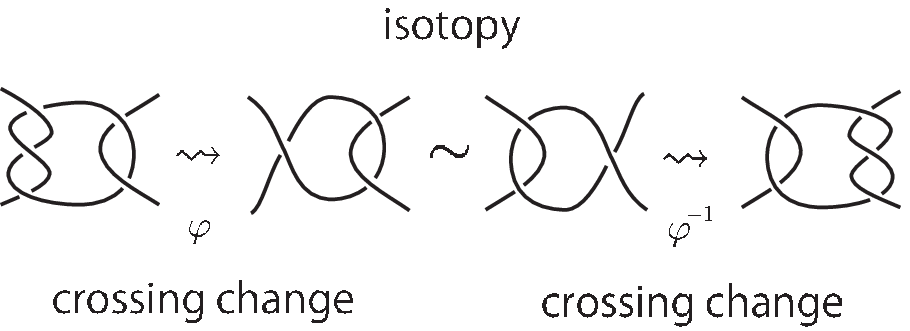}
\caption{The factorization of the knot mutation.}
\label{mufact}
\end{center}
\end{figure}
According to this process, we define $\mu$ to be the map obtained by the process as described in {\sc Figure}~\ref{mutationplug}.
\begin{figure}[htbp]
\begin{center}
\includegraphics{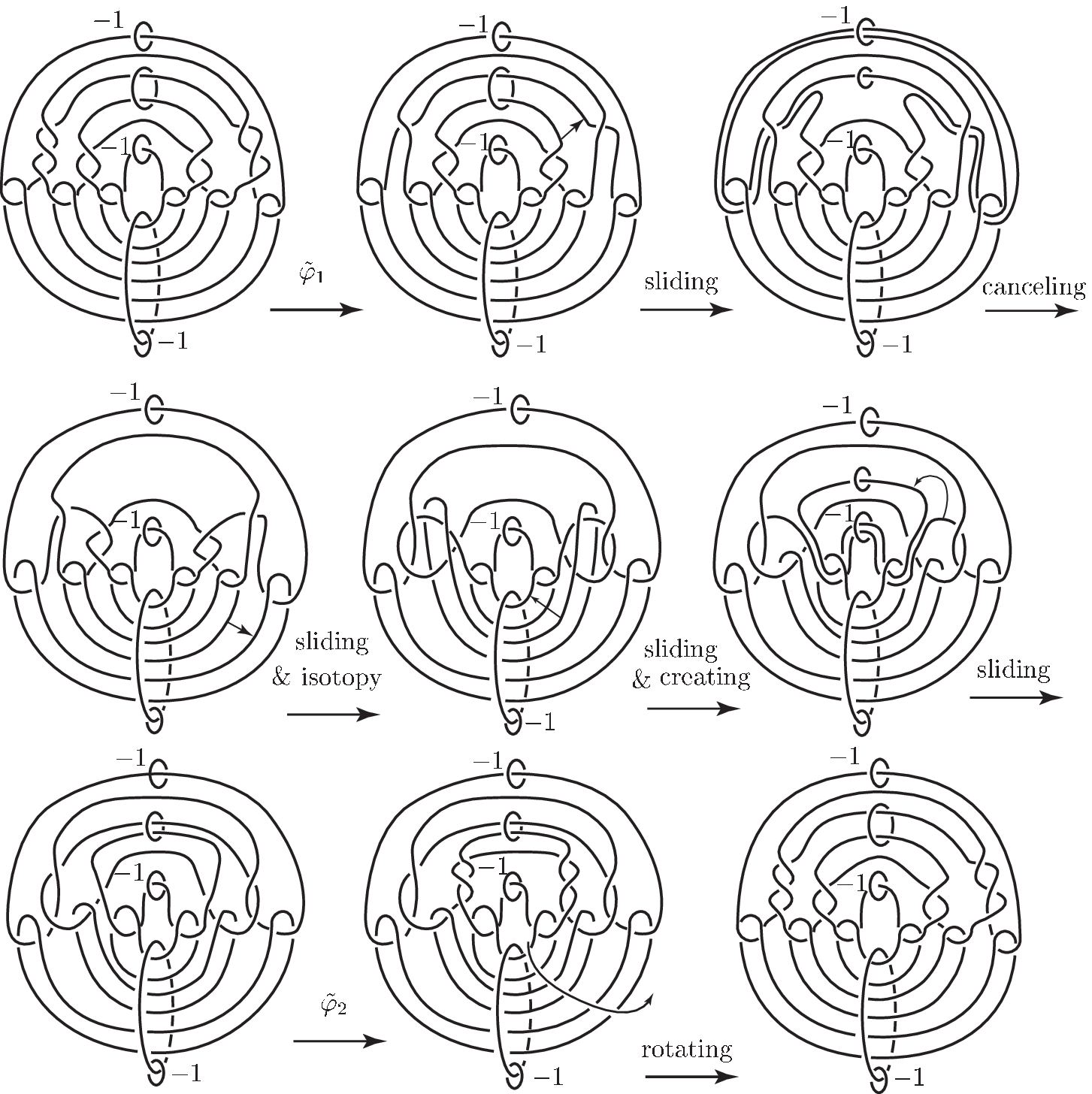}
\caption{The definition of $\mu$.}
\label{mutationplug}
\end{center}
\end{figure}
Here $\tilde{\varphi}_1,\tilde{\varphi}_2:\partial M\to \partial M$ are maps obtained by performing locally $\varphi,\varphi^{-1}$ on $\partial M$.


{\bf Proof of Theorem~\ref{Mmutant}.}
Let $K,K'$ be a mutant pair.
We find an embedding $M\hookrightarrow V_K$.
Let $D$ be a knot diagram of the knot $K$ containing the local tangle of the right in {\sc Figure}~\ref{knotm}.
For example, the first picture in {\sc Figure}~\ref{mdiff} is such a diagram.
We move the local tangle surrounded by the broken line to a bottom position by some isotopy (the second picture).
The resulting diagram gives a plate presentation with keeping the local picture in the bottom (the third picture).

We prove that $(M,\mu)$ is a twist between knot-surgeries for mutant pair $K$ and $K'$
by illustrating the case of $V_{KT}\leadsto V_C$ in {\sc Figure}~\ref{KTC}, where $KT$ is the Kinoshita-Terasaka knot and $C$ is the Conway knot.

By keeping track of the processes in {\sc Figure}~\ref{mutationplug}, the square $\mu^2$ is the two times of the last move in {\sc Figure}~\ref{mutationplug}.
This means a $360^{\circ}$ rotation of $\partial M$ along the torus.
This is homotopic to the identity.
\hfill\qed
\begin{figure}\begin{center}\includegraphics{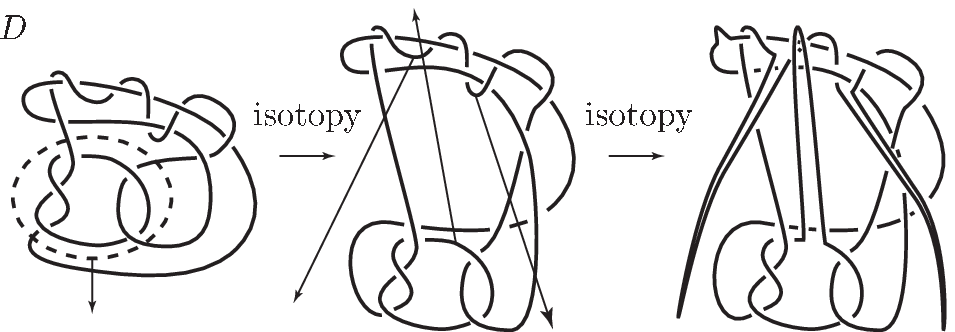}\caption{Moving the local tangle with respect to the mutant move.}\label{mdiff}\end{center}\end{figure}
\begin{figure}[htbp]
\begin{center}
\includegraphics{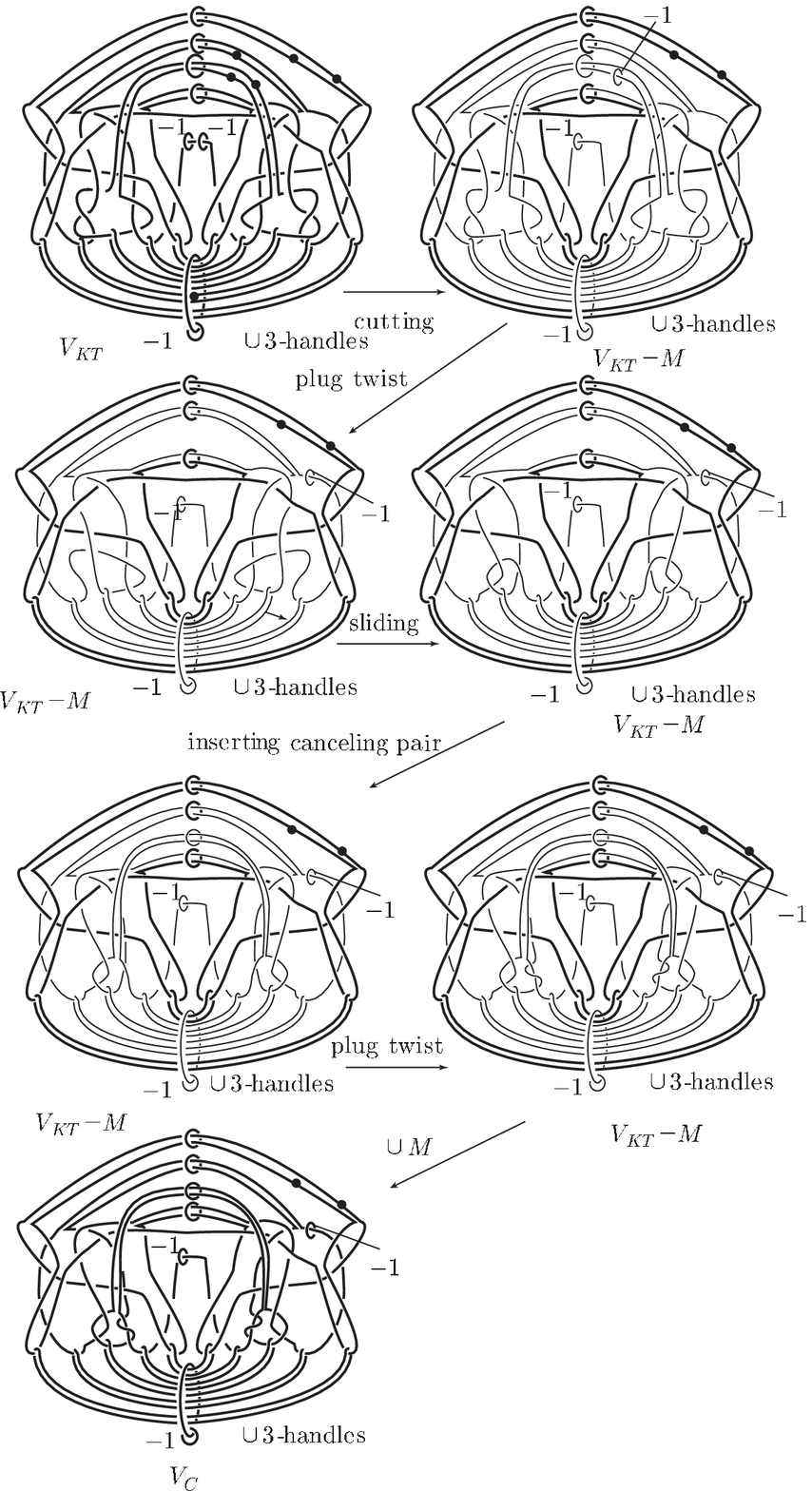}
\caption{The performance $V_{KT}\leadsto V_{KT}-M\leadsto (V_{KT}-M)\cup_{\mu}M=V_C$. The fine curve presents the removed handles for $V_{KT}-M$.}
\label{KTC}
\end{center}
\end{figure}

{\bf Proof of Proposition~\ref{Mmutant2}.}
The first picture in {\sc Figure}~\ref{muinverse} presents the untwisted double $D(M):=M\cup_{\text{id}}(-M)$.
We can easily check the diffeomorphism $D(M)\cong \#^3S^2\times S^2$ by handle calculus.
Removing $M$ in $D(M)$, regluing by $\mu$, we get the next picture in {\sc Figure}~\ref{muinverse}.
The intersection form of the twisted double $D_\mu(M)$ is isomorphic to $\oplus^3H$.
Thus, by using Boyer's result in \cite{[B]}, $\mu$ can extend to a self-homeomorphism $M\to M$.
\begin{figure}[htbp]
\begin{center}
\includegraphics{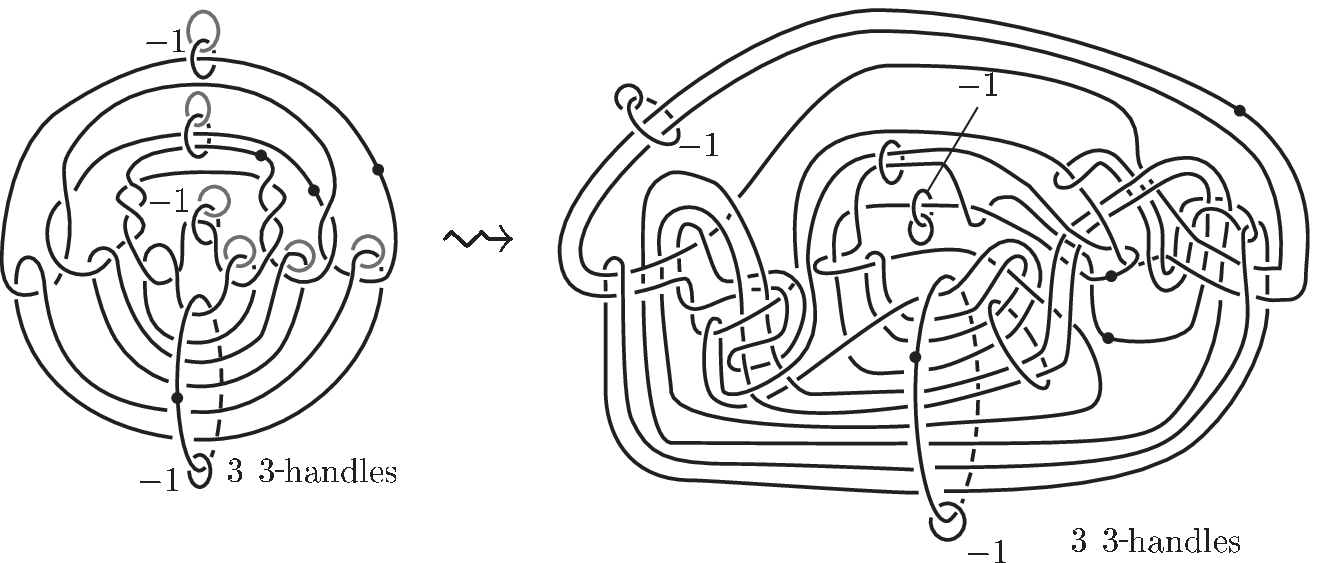}
\caption{$D(M)\leadsto D_{\mu}(M)$ (via the local move $(M,\mu)$).}
\label{muinverse}
\end{center}
\end{figure}
\hfill\qed

Here we define $M_0$ to be $M$ with a $-1$-framed 2-handle deleted (the left of {\sc Figure}~\ref{Mzerofig}).
The boundary map $\mu_0:\partial M_0\to \partial M_0$ is naturally induced from the map $\mu$, because
the $-1$-framed 2-handle in $M$ is fixed via the map $\mu$.
The diffeomorphism $D_{\text{id}}(M_0)\cong \#^2S^2\times S^2\#S^3\times S^1$ and
the homeomorphism $D_{\mu_0}(M_0)\simeq \#^2S^2\times S^2$ hold due to easy calculation.
\begin{figure}[htbp]
\begin{center}\includegraphics{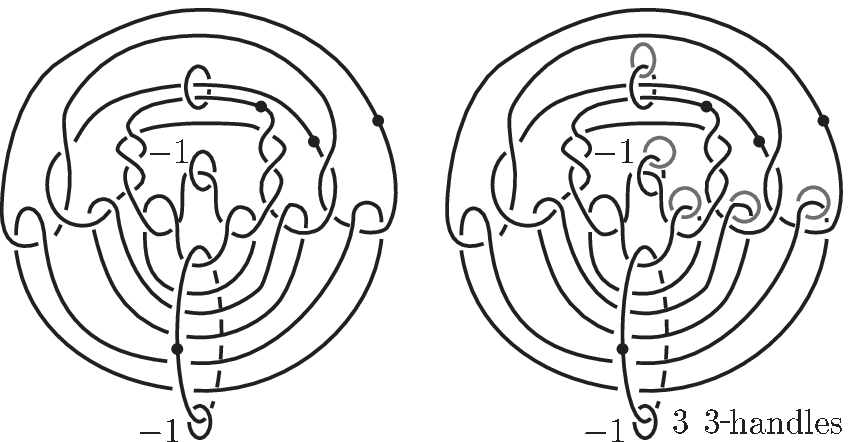}\caption{$M_0$ and $D(M_0)=\#^3S^2\times S^2\#S^3\times S^1$.}\label{Mzerofig}\end{center}
\end{figure}

{\bf Proof of Proposition~\ref{twisteddouble}.}
The outmost (Hopf-linked) pair of $-1$-framed 2-handle and $0$-framed 2-handle in {\sc Figure}~\ref{muinverse}
can be moved to the parallel position of the other Hopf-linked pair by several handle slides.
Such handle slides are indicated in {\sc Figure}~\ref{doublehan}.
Hence, the pair can be removed as one Hopf link component with both framings $0$.
See the bottom row in {\sc Figure}~\ref{doublehan}.
The same deformation is seen in Fig.15 in \cite{[T]}.
The remaining part is $D_{\mu_0}(M_0)$.
\begin{figure}[htbp]
\begin{center}\includegraphics{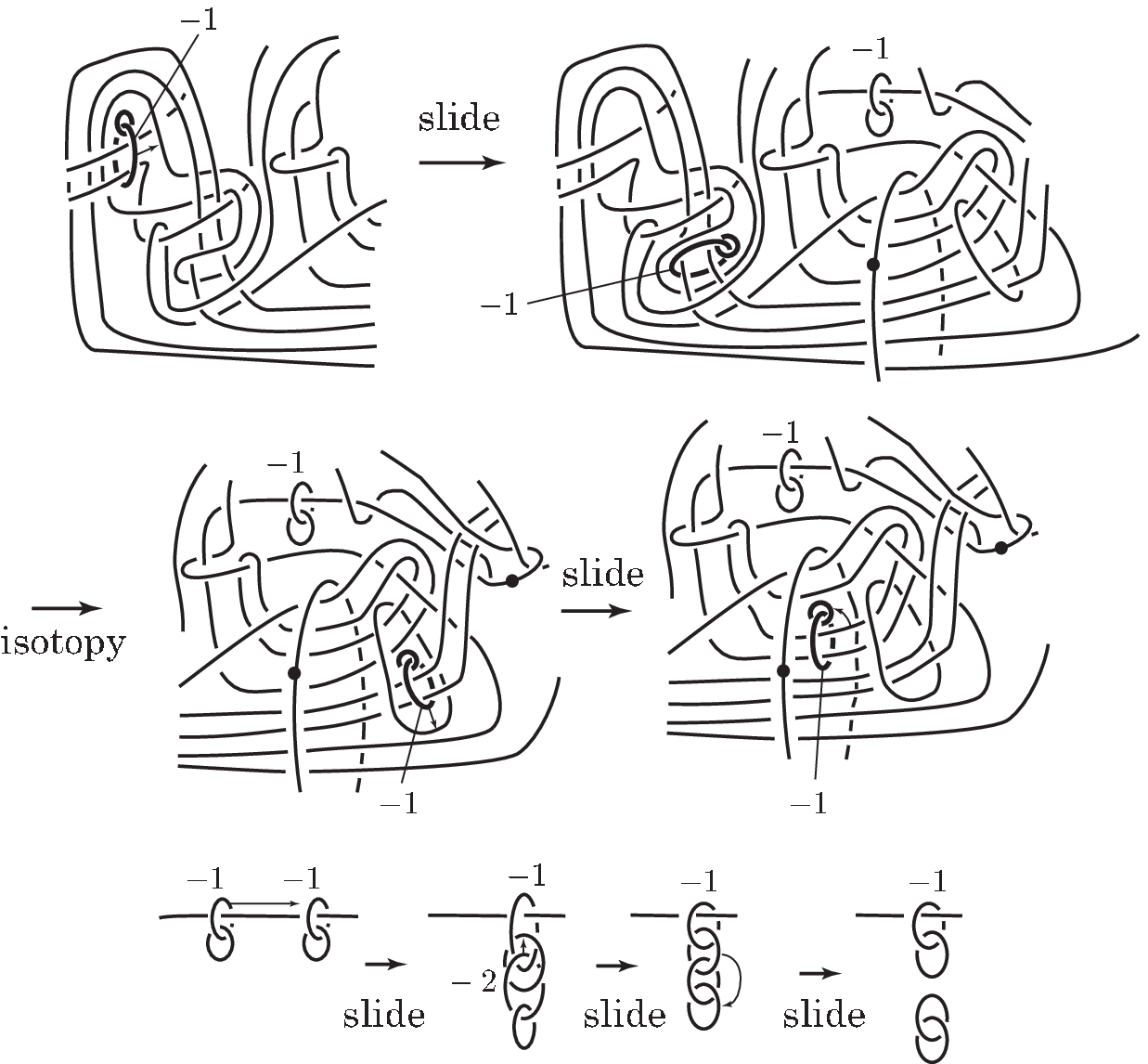}\caption{To move a pair of 2-handles to the position of the other pair.}\label{doublehan}\end{center}
\end{figure}
\hfill\qed

\end{document}